\documentclass{amsart}

\usepackage{amsmath,amssymb,amsfonts,amsthm}
\usepackage{enumerate}
\usepackage{array}
\usepackage{mathrsfs}
\usepackage{csquotes}
\usepackage{fullpage}
\usepackage{graphicx}
\usepackage[all]{xy}
\usepackage{ latexsym }
\usepackage{mathtools}
\usepackage{tikz}

\usepackage{color}
\definecolor{Chocolat}{rgb}{0.36, 0.2, 0.09}
\definecolor{BleuTresFonce}{rgb}{0.215, 0.215, 0.36}

\usepackage[colorlinks,final,hyperindex]{hyperref}
\hypersetup{citecolor=BleuTresFonce, linkcolor=Chocolat}

\newcommand{\Grav}{\mathcal{G}\textit{rav}\,}

\newcommand{\M}{\mathcal{M}}
\newcommand{\Mbar}{\overline{\mathcal{M}}}
\newcommand{\Mdelta}{\mathcal{M}^\delta}
\newcommand{\K}{\mathcal{K}}

\newcommand{\Q}{\mathbb{Q}}
\newcommand{\C}{\mathbb{C}}

\renewcommand{\leq}{\leqslant}
\renewcommand{\geq}{\geqslant}

\newcommand{\isom}{\simeq}
\newcommand{\tr}{\mathfrak{t}}
\newcommand{\di}{\mathfrak{d}}
\newcommand{\calC}{\mathcal{C}}

\newtheorem*{thma}{Theorem A}
\newtheorem*{thmb}{Theorem B}

\theoremstyle{theorem}
\newtheorem{thm}{Theorem}[section]
\newtheorem{prop}[thm]{Proposition}
\newtheorem{lem}[thm]{Lemma}
\newtheorem{coro}[thm]{Corollary}

\theoremstyle{remark}
\newtheorem{rem}[thm]{Remark}
\newtheorem{ex}[thm]{Example}

\theoremstyle{definition}
\newtheorem{defi}[thm]{Definition}

\title{Brown's moduli spaces of curves and the gravity operad}



\thanks{B.V. was supported by the ANR SAT grant. }

\author{Cl\'{e}ment Dupont}
\address{ Institut Montpelli\'{e}rain Alexander Grothendieck, Universit\'{e} de Montpellier, CNRS, UMR 5149,
Place Eug\`{e}ne Bataillon, 
34090 Montpellier, France.}
\email{clement.dupont@umontpellier.fr}

\author{Bruno Vallette}
\address{Laboratoire Analyse, G\'eom\'etrie et Applications, Universit\'e Paris 13, Sorbonne Paris Cit\'e, CNRS, UMR 7539, 93430 Villetaneuse, France.}
\email{vallette@math.univ-paris13.fr}

\begin{document}

\begin{abstract}
This paper is built on the following observation: the purity of the mixed Hodge structure on the cohomology of Brown's moduli spaces is essentially equivalent to the freeness of the dihedral operad underlying the gravity operad. We prove these two facts by relying on both the geometric and the algebraic aspects of the problem: the complete geometric description of the cohomology of Brown's moduli spaces and the coradical filtration of cofree cooperads. This gives a conceptual proof of an identity of Bergstr\"om--Brown which expresses the Betti numbers of Brown's moduli spaces via the inversion of a generating series. This also generalizes the Salvatore--Tauraso theorem on the nonsymmetric Lie operad. 
\end{abstract}

\maketitle

\setcounter{tocdepth}{1}
\tableofcontents

\section*{Introduction}

	The moduli space of genus zero smooth curves with~$n$ marked points, denoted by~$\M_{0,n}$, is a classical object in algebraic geometry, as well as its Deligne--Mumford--Knudsen compactification~$\Mbar_{0,n}$, which parametrizes stable genus zero curves with~$n$ marked points. In~\cite{brownPhD}, Brown introduced a \enquote{partial compactification}
	$$\M_{0,n}\subset\Mdelta_{0,n}\subset\Mbar_{0,n}$$
	in order to prove a conjecture of Goncharov and Manin~\cite{goncharovmanin} on the relation between certain period integrals on~$\Mbar_{0,n}$ and multiple zeta values.

\smallskip

	The homology groups of the moduli spaces~$\M_{0,n}$, as well as those of the compactified moduli spaces~$\Mbar_{0,n}$, assemble to form two operads, respectively called the gravity and hypercommutative operads by Getzler. These two operads are Koszul dual in the sense of the Koszul duality of operads; see Getzler ~\cite{getzlermodulispacesgenuszero} and Ginzburg and Kapranov ~\cite{ginzburgkapranov}. As pointed out by Getzler, this is very much related to the purity of the mixed Hodge structures on the cohomology groups under consideration. This implies that the exponential generating series encoding the Betti numbers of~$\M_{0,n}$ and~$\Mbar_{0,n}$ are inverse to one another.

\smallskip

	A similar identity was conjectured by Bergstr\"om and Brown in~\cite{bergstrombrown}: the ordinary generating series encoding the Betti numbers of the moduli spaces~$\M_{0,n}$ and~$\M^\delta_{0,n}$ should be inverse to one another. More precisely, it is showed how such a relation can be derived from a more conceptual fact: the purity of the mixed Hodge structure on the cohomology groups of Brown's moduli spaces. This is the first result of the present paper. 

	\begin{thma}
	For every integers~$k$ and~$n$, the mixed Hodge structure on the cohomology group~$H^k(\Mdelta_{0,n})$ is pure Tate of weight~$2k$.
	\end{thma}
	
	This theorem has the following straightforward consequences:
	\begin{enumerate}[--]
	\item the cohomology algebra of Brown's moduli space~$\Mdelta_{0,n}$ embeds into that of the moduli space~$\M_{0,n}$ (Corollary~\ref{coroinj});
	\item there is a recursive formula for the Betti numbers of~$\Mdelta_{0,n}$, conjectured in Bergstr\"om and Brown ~\cite{bergstrombrown} (Corollary~\ref{corobetti});
	\item Brown's moduli spaces~$\Mdelta_{0,n}$ are formal topological spaces in the sense of rational homotopy theory (Corollary~\ref{coroformality}).
	\end{enumerate}
	
	It turns out that the purity of the mixed Hodge structure of Theorem A can be equivalently interpreted in the following operadic terms.

	\begin{thmb}
	The dihedral gravity operad is free. Its space of generators in arity~$n$ and degree~$k$ is (non-canonically) isomorphic to the homology group~$H_{k+n-3}(\Mdelta_{0,n})$.
	\end{thmb}
	
	We introduce here the new notion of a dihedral operad, which faithfully takes into account the dihedral symmetry of Brown's moduli spaces. Such a notion forgets almost all the symmetry properties of a cyclic operad, except for the dihedral structure. Theorem B can also be viewed as a kind of nonsymmetric analog of the Koszul duality between the gravity and the hypercommutative operad, since a free operad is Koszul, its dual being a nilpotent operad. We prove it by introducing a combinatorial filtration on the cohomology groups of the spaces $\M_{0,n}$, and identifying it with the coradical filtration of the dihedral gravity cooperad.
	
	\smallskip
	
	The problem of studying whether the nonsymmetric operad underlying a given operad is free is not new. In~\cite{salvatoretauraso}, Salvatore and Tauraso proved that the nonsymmetric operad underlying the  operad of Lie algebras is free. This result is actually the top dimensional part of Theorem B. Thus, the geometric methods developed throughout this paper provide us with a new proof of (a dihedral enhancement of) the theorem of Salvatore and Tauraso. 
	
	\smallskip
	
	Note that in the preprint~\cite{almpetersen} (which appeared on the arXiv one day after the present article), Alm and Petersen give independent proofs of Theorem A and Theorem B. Their proofs rely on an explicit basis for the gravity cooperad, and a construction of Brown's moduli spaces in terms of blow-ups and deletions. The freeness of the (nonsymmetric) gravity operad has been used in~\cite{almBV} to study an exotic $A_\infty$-structure on Batalin--Vilkovisky algebras.

	\subsection*{Layout.} 
	The first section deals with the various combinatorial objects and notions of operads used in this text. In the second section, we introduce the moduli spaces of curves~$\M_{0,n}$ and~$\Mbar_{0,n}$, as well as the notion of mixed Hodge structure. The study of Brown's moduli spaces~$\M^\delta_{0,n}$ and the dihedral gravity cooperad fills the third section. The fourth section contains the proofs of Theorems A and B and their corollaries. 
	
	\subsection*{Conventions.} Throughout the paper, the field of coefficients is the field~$\Q$ of rational numbers. For a topological space~$X$, we simply denote by~$H_\bullet(X)$ and~$H^\bullet(X)$ the (co)homology groups of~$X$ with rational coefficients. We work with graded vector spaces and switch between the homological convention (with degrees as subscripts) and the cohomological convention (with degrees as superscripts), the two conventions being linear dual to one another.
	
	\subsection*{Acknowledgements.} We would like to express our sincere appreciation to Johan Alm, Francis Brown and Dan Petersen for useful discussions, and to the anonymous referee for their suggestions, which helped improve the clarity of this article. We would like to thank the Max-Planck-Institut f\"{ur} Mathematik (where the first author was holding a position and where the second author came during several visits) and the University Nice Sophia Antipolis (vice versa) for the excellent working conditions.

\section{Freeness criteria for dihedral cooperads}

	The purpose of this first section is to recall the various notions of operads (classical, cyclic, nonsymmetric, cyclic nonsymmetric) and to introduce a new one (dihedral operad) which suits the geometry of Brown's moduli spaces. We first describe the combinatorial objects (trees and polygon dissections) involved in the proof of the results of the paper. In the end of this section, we prove two freeness criteria for dihedral cooperads, one based on their cobar construction and the other based on their coradical filtration.

	\subsection{Dissections of polygons and trees}\label{sec:Diss}

		\begin{defi}[Structured sets] Let~$S$ be a finite set of cardinality~$n$.
		\begin{enumerate}[--]
		
		\item A \emph{basepoint}~$\rho$ on~$S$ is a map~$\rho : \{*\} \to S$. A pair~$(S,\rho)$ is called a \emph{pointed set}.

		\item A \emph{total order}~$\omega$ on~$S$ is a bijection between~$S$ and the set~$\{1,\ldots,n\}$. There are~$n!$ total orders on~$S$. A pair~$(S,\omega)$ is called a \emph{totally ordered set}. By convention, we view a totally ordered set as a pointed set, the basepoint being the maximal element.

		\item A \emph{cyclic structure}~$\gamma$ on~$S$ is an identification of~$S$ with the edges of an oriented~$n$-gon, modulo rotations. There are~$\frac{n!}{n}=(n-1)!$ cyclic structures on~$S$. A pair~$(S,\gamma)$ is called a \emph{cyclic set}.

		\item A \emph{dihedral structure}~$\delta$ on~$S$ is an identification of~$S$ with the edges of an unoriented~$n$-gon, modulo dihedral symmetries. There are~$\frac{n!}{2n}=\frac{1}{2}(n-1)!$ dihedral structures on~$S$. A pair~$(S,\delta)$ is called a \emph{dihedral set}.

		\end{enumerate}
		\end{defi}
		
		In the sequel, we will identify a dihedral set~$(S,\delta)$ with an unoriented polygon with its edges decorated by~$S$ in the dihedral order prescribed by~$\delta$.

		\begin{defi}[Chords and dissections]
		Let~$(S,\delta)$ be a dihedral set.
		 \begin{enumerate}[--]
		\item	A \emph{chord} of~$(S,\delta)$ is an unordered pair of non-consecutive vertices of the underlying unoriented polygon. 

		\item A \emph{dissection}~$\di$ of~$(S,\delta)$ is a (possibly empty) set of non-crossing chords. 
		The refinement of dissections endows them with a poset structure:
		$$ \di\leq \di' \quad \text{if} \quad \di \subset \di'\ ,~$$
		in which the smallest element is the empty dissection.
		We denote by~$\mathsf{Diss}(S,\delta)$ the poset of dissections of~$(S,\delta)$, and by~$\mathsf{Diss}_k(S,\delta)$ the subset consisting of dissections with~$k$ chords.
		\end{enumerate}
		\end{defi}
		
		For a dissection~$\di\in\mathsf{Diss}(S,\delta)$, we denote by~$P(\di)$ the set of sub-polygons that it defines, see Figure~\ref{figuredissection}. If~$\di$ is in~$\mathsf{Diss}_k(S,\delta)$, then~$P(\di)$ has cardinality~$k+1$.	 A sub-polygon~$p\in P(\di)$ corresponds to a dihedral set that we denote by~$(E(p),\delta(p))$, where~$E(p)$ consists of edges and chords of the polygon~$(S,\delta)$.
		
		\begin{figure}[h!!]
		\def\svgwidth{.28\textwidth}
		\begin{center}
\begin{tikzpicture}[scale=0.3]

\draw[thick] (0:10) -- (30:10) -- (60:10) -- (90:10) -- (120:10) -- (150:10) -- (180:10) 
-- (210:10) -- (240:10) -- (270:10) -- (300:10) -- (330:10) -- (360:10);

\draw[thick, blue] (60:10) -- (300:10) node[midway, below left] {\scalebox{1}{$\mathbf{c_3}$}};
\draw[thick, blue] (90:10) -- (270:10) node[midway, below left] {\scalebox{1}{$\mathbf{c_2}$}};
\draw[thick, blue] (150:10) -- (270:10) node[midway, left] {\scalebox{1}{$\mathbf{c_1}$}};

\node[purple] at (10:7.7) {\scalebox{1}{$\mathbf{p_3}$}};
\node[purple] at (28:2.7) {\scalebox{1}{$\mathbf{p_2}$}};
\node[purple] at (140:4) {\scalebox{1}{$\mathbf{p_1}$}};
\node[purple] at (210:8) {\scalebox{1}{$\mathbf{p_0}$}};
\end{tikzpicture}
		\end{center}
		\caption{A dissection~$\di=\{c_1,c_2,c_3\}$, with the set of sub-polygons~$P(\di)=\{p_0,p_1,p_2,p_3\}$.}\label{figuredissection}
		\end{figure}
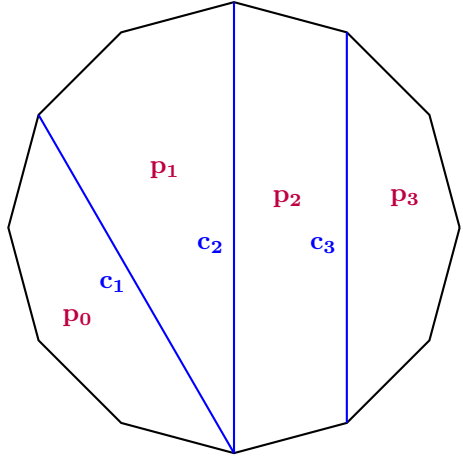	
		
		\begin{defi}[Trees]\label{def-Tree}
		A \emph{tree} is a finite graph with no cycle. The contraction of internal edges endows trees with a poset structure: we set~$\tr\leq \tr'$ if the tree~$\tr$ can be obtained from the tree~$\tr'$ by contracting some internal edges. If the number of external vertices is fixed, the minimal element of this poset is the only tree with zero internal edge, called a \emph{corolla}. By looking at the possible structures on the set of external vertices of a tree, we get different posets:
		\begin{enumerate}[--]
		\item the poset~$\mathsf{Tree}(S)$ of trees with external vertices labeled by~$S$;
		\item the poset~$\mathsf{RTree}(S,\rho)$ of rooted trees with external vertices labeled by~$S$, the root being labeled by the basepoint~$\rho$;
		\item the poset~$\mathsf{PRTree}(S,\omega)$ of planar rooted trees with external vertices labeled by~$S$ in the total order~$\omega$, the root being labeled by the maximal element;	
		\item the poset~$\mathsf{PTree}(S,\gamma)$ of planar trees with external vertices labeled by~$S$ in the cyclic order~$\gamma$;		
		\item the poset~$\mathsf{DTree}(S,\delta)$ of dihedral trees (trees embedded in an unoriented plane) with external vertices labeled by~$S$ in the dihedral order~$\delta$.
		\end{enumerate}
		All these posets are graded by the number of internal edges of the trees.
		\end{defi}
		
		For a tree~$\tr$, we denote its set of vertices by~$V(\tr)$. For each vertex~$v \in V(\tr)$, we denote  its set of adjacent edges by~$E(v)$. Notice that if~$\tr$ is a rooted tree then we get a pointed set~$(E(v),\rho(v))$; if~$\tr$ is a planar rooted tree then we get a totally ordered set~$(E(v),\omega(v))$; if~$\tr$ is a planar tree then we get a cyclic set~$(E(v),\gamma(v))$; if~$\tr$ is a dihedral tree then we get a dihedral set~$(E(v),\delta(v))$. We refer the reader to ~\cite[Section~C.4]{lodayvallettebook} for more details on the notions related to trees. 
				
		\begin{lem}\label{lem:Diss-Tree}
		The graded poset~$\mathsf{Diss(S, \delta)}$ of dissections of a polygon~$(S,\delta)$ and the graded poset~$\mathsf{DTree}(S, \delta)$ of dihedral trees labeled by the dihedral set~$(S,\delta)$ are isomorphic.
		\end{lem}		
				
		\begin{figure}[h!!]
		\def\svgwidth{.30\textwidth}
		\begin{center}
\begin{tikzpicture}[scale=0.3]

\draw[thick] (0:10) -- (30:10) -- (60:10) -- (90:10) -- (120:10) -- (150:10) -- (180:10) 
-- (210:10) -- (240:10) -- (270:10) -- (300:10) -- (330:10) -- (360:10);

\draw[thick, blue] (60:10) -- (300:10) ;
\draw[thick, blue] (90:10) -- (270:10) ;
\draw[thick, blue] (150:10) -- (270:10) ;

\node[purple] at (0:7.7) {\scalebox{2.5}{$\bullet$}};
\node[purple] at (0:2) {\scalebox{2.5}{$\bullet$}};
\node[purple] at (140:4) {\scalebox{2.5}{$\bullet$}};
\node[purple] at (210:8) {\scalebox{2.5}{$\bullet$}};

\draw[purple, very thick] (0:7.7) -- (0:2) -- (140:4) -- (210:8);

\draw[purple, very thick] (0:7.7) -- (15:11);
\draw[purple, very thick] (0:7.7) -- (-15:11);
\draw[purple, very thick] (0:7.7) -- (45:11);
\draw[purple, very thick] (0:7.7) -- (-45:11);

\draw[purple, very thick] (0:2) -- (75:11);
\draw[purple, very thick] (0:2) -- (-75:11);

\draw[purple, very thick] (140:4) -- (105:11);
\draw[purple, very thick] (140:4) -- (135:11);

\draw[purple, very thick] (210:8) -- (165:11);
\draw[purple, very thick] (210:8) -- (195:11);
\draw[purple, very thick] (210:8) -- (225:11);
\draw[purple, very thick] (210:8) -- (255:11);

\end{tikzpicture}

		\end{center}
		\caption{The isomorphism between polygon dissections and dihedral trees.}\label{figuredissectiontree}
		\end{figure}
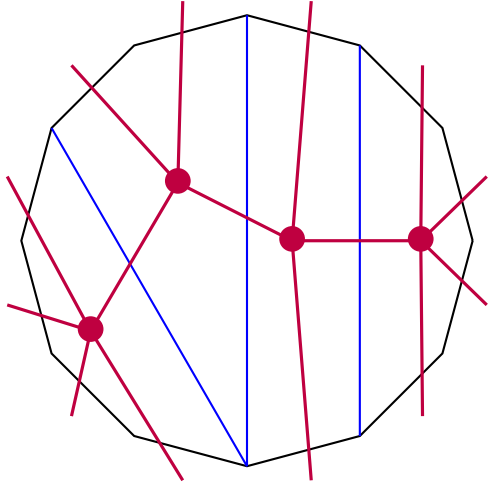					
				
		\begin{proof}
		Let us describe the isomorphism~$\mathsf{Diss(S, \delta)} \to \mathsf{DTree}(S, \delta)$. Given a dissection~$\di\in\mathsf{Diss}(S,\delta)$, one considers its \enquote{dual graph}~$\tr$: each sub-polygon~$p\in P(\di)$ gives rise to a vertex~$v\in V(\tr)$ of the tree~$\tr$ and  each edge of this polygon gives rise to an edge of the tree, see Figure~\ref{figuredissectiontree}. The tree~$\tr$ is naturally a dihedral tree, and it is straightforward to check that this defines a bijection between~$\mathsf{Diss}(S,\delta)$ and~$\mathsf{DTree}(S,\delta)$. Under this bijection, removing a chord from the dissection corresponds to contracting internal edges of trees, hence we get an isomorphism of posets, which respects the grading by construction. 
		\end{proof}

	\subsection{Dihedral operads}\label{sec:DiOp} In this section, we recall the classical  notions of operads and we introduce a new one, the notion of dihedral operad, which suits the geometrical problem studied here. 
			We work in the general setting of an abelian  symmetric monoidal category~$(\mathsf{A}, \otimes)$ such that the monoidal product preserves coproducts.  In the next section and later on, we will specify the category~$\mathsf{A}$ to be the category  of graded mixed Hodge structures. 
			
		\begin{defi}[Categories of structured sets]\label{defimodules}
			We consider the following categories of structured sets.
			\begin{enumerate}[--]
			\item The category~$\mathsf{Bij}$ of finite sets~$S$ and bijections. 
			\item The category~$\mathsf{Bij}_*$ of pointed sets~$(S, \rho)$ and bijections respecting the basepoint. 
			\item The category~$\mathsf{Ord}_*$ of totally ordered sets~$(S,\omega)$ and bijections respecting the total order. 
			\item The category~$\mathsf{Cyc}$ of cyclic sets~$(S, \gamma)$ and bijections respecting the cyclic order. 
			\item The category~$\mathsf{Dih}$ of dihedral sets~$(S, \delta)$ and bijections respecting the dihedral structure. 
		\end{enumerate}
		\end{defi}

		The forgetful functors between the various categories of structured sets assemble as a commutative diagram
		$$
		\xymatrix{
		\textsf{Dih} \ar[d]& \textsf{Cyc} \ar[l] & \ar[l]\textsf{Ord}_*\ar[d] \\
		 \textsf{Bij} && \ar[ll] \textsf{Bij}_*
		}
		$$
		where the functor~$\mathsf{Ord}_*\rightarrow\mathsf{Bij}_*$ picks the maximal element as basepoint.

		In each case, we consider the category of functors from these categories to the category~$\mathsf{A}$, for instance~$\mathcal{M} : \mathsf{Bij} \to \mathsf{A}$, that we respectively call the category of ~$\mathsf{Bij}$-modules,~$\mathsf{Bij}_*$-modules,~$\mathsf{Ord}_*$-modules,~$\mathsf{Cyc}$-modules, and~$\mathsf{Dih}$-modules. We denote them respectively by~$\mathsf{Bij}$-$\mathsf{Mod}$,~$\mathsf{Bij}_*$-$\mathsf{Mod}$,~$\mathsf{Ord}_*$-$\mathsf{Mod}$,~$\mathsf{Cyc}$-$\mathsf{Mod}$, and~$\mathsf{Dih}$-$\mathsf{Mod}$. We then get a commutative diagram of forgetful functors.
		
		$$
		\xymatrix{
		\textsf{Dih-Mod} \ar[r] & \textsf{Cyc-Mod} \ar[r] & \textsf{Ord}_*\textsf{-Mod} \\
		 \textsf{Bij-Mod} \ar[rr]\ar[u]&& \ \textsf{Bij}_*\textsf{-Mod} \ar[u]\ .
		}
		$$
		
		In the next definition we are using tensor products labeled by sets; see~\cite[Section~$5.1.14$]{lodayvallettebook} for more details on this notion.
		
		\begin{defi}[Monads of trees]\label{defimonads}
		We consider the following monads of trees. 
		\begin{enumerate}[--]
		\item The monad~$\mathbb{T} : \mathsf{Bij}\text{-}\mathsf{Mod}\to \mathsf{Bij}\text{-}\mathsf{Mod}$ 	is defined via trees:
			$$\mathbb{T}\mathcal{M}(S):=\bigoplus_{\mathfrak{t}\in \mathsf{Tree}(S)} \left( \bigotimes_{v\in V(\mathfrak{t})} \mathcal{M}(E(v)) \right).$$	 
		\item The monad~$\mathbb{RT} : \mathsf{Bij}_*\text{-}\mathsf{Mod}\to \mathsf{Bij}_*\text{-}\mathsf{Mod}$ 	is defined via  rooted trees:
			$$\mathbb{RT}\mathcal{M}(S, \rho):=\bigoplus_{\mathfrak{t}\in \mathsf{RTree}(S, \rho)} \left( \bigotimes_{v\in V(\mathfrak{t})} \mathcal{M}(E(v),\rho(v)) \right).$$	 
		\item The monad~$\mathbb{PRT} : \mathsf{Ord}_*\text{-}\mathsf{Mod}\to \mathsf{Ord}_*\text{-}\mathsf{Mod}$ 	is defined via  planar rooted trees:
			$$\mathbb{PRT}\mathcal{M}(S, \omega):=\bigoplus_{\mathfrak{t}\in \mathsf{PRTree}(S, \omega)} \left( \bigotimes_{v\in V(\mathfrak{t})} \mathcal{M}(E(v),\omega(v)) \right).$$
		\item The monad~$\mathbb{PT} : \mathsf{Cyc}\text{-}\mathsf{Mod}\to \mathsf{Cyc}\text{-}\mathsf{Mod}$ 	is defined via  planar trees:
			$$\mathbb{PT}\mathcal{M}(S, \gamma):=\bigoplus_{\mathfrak{t}\in \mathsf{PTree}(S, \gamma)} \left( \bigotimes_{v\in V(\mathfrak{t})} \mathcal{M}(E(v),\gamma(v)) \right).$$	 
		\item The monad~$\mathbb{DT} : \mathsf{Dih}\text{-}\mathsf{Mod}\to \mathsf{Dih}\text{-}\mathsf{Mod}$ 	is defined via  dihedral trees:
			$$\mathbb{DT}\mathcal{M}(S, \delta):=\bigoplus_{\mathfrak{t}\in \mathsf{DTree}(S, \delta)} \left( \bigotimes_{v\in V(\mathfrak{t})} \mathcal{M}(E(v),\delta(v)) \right).$$	 
		\end{enumerate}
		The composition law of these monads, e.g.~$\mathbb{T}\circ\mathbb{T}\rightarrow\mathbb{T}$, is given by substitution of trees, and the unit, e.g.~$1\rightarrow\mathbb{T}$, is given by the inclusion into the direct summand indexed by corollas. See~\cite[Section~$5.6.1$]{lodayvallettebook} for more details.
		\end{defi}
		
		\begin{rem}\label{remmonad}
		In the above commutative diagram, the horizontal forgetful functors commute with the respective monads: the forgetful functor $\mathsf{Dih}\text{-}\mathsf{Mod}\rightarrow\mathsf{Cyc}\text{-}\mathsf{Mod}$ commutes with $\mathbb{DT}$ and $\mathbb{PT}$; the forgetful functor $\mathsf{Cyc}\text{-}\mathsf{Mod}\rightarrow \mathsf{Ord}_*\text{-}\mathsf{Mod}$ commutes with $\mathbb{PT}$ and $\mathbb{PRT}$; the forgetful functor $\mathsf{Bij}\text{-}\mathsf{Mod} \rightarrow \mathsf{Bij}_*\text{-}\mathsf{Mod}$ commutes with $\mathbb{T}$ and $\mathbb{RT}$. There is no corresponding statement for the vertical forgetful functors. 
		\end{rem}
	
		\begin{defi}[Types of operads]\label{defioperads}
		An \emph{operad} (resp. a \emph{cyclic operad}, a \emph{nonsymmetric  operad}, a \emph{nonsymmetric cyclic operad}, and a \emph{dihedral operad}) is an algebra over the monad~$\mathbb{RT}$ of rooted trees (resp. 
		the monad~$\mathbb{T}$ of  trees, 
		the monad~$\mathbb{PRT}$ of planar rooted  trees, 
		the monad~$\mathbb{PT}$ of planar  trees, 
		and the monad~$\mathbb{DT}$ of dihedral  trees).
		\end{defi}
		
		\begin{rem}
		In the rest of this article, we will always assume that all finite sets $S$ have cardinality $n\geq 3$. This is more convenient for our geometric purposes, since the moduli spaces~$\M_{0,S}$ and~$\Mbar_{0,S}$ are only defined for those sets, and also to avoid speaking of polygons with $2$ sides. The operads that we manipulate are then \emph{non-unital operads}.
		\end{rem}
			
		The aforementioned diagram of categories 
	gives rise to the following forgetful functors between the categories of operads 
		$$
		\xymatrix{
		\textsf{Dih-Op} \ar[r] & \textsf{ns-Cyc-Op} \ar[r] & \textsf{ns-Op} \\
		 \textsf{Cyc-Op} \ar[rr]\ar[u]&& \ \textsf{Op} \ar[u]\ .
		}
		$$
		
		\begin{rem}\label{remfreeoperad}
		In view of Remark \ref{remmonad}, the free dihedral operad, the free nonsymmetric cyclic operad and the free nonsymmetric operad on a given $\mathsf{Dih}$-module have the same underlying nonsymmetric operad.
		\end{rem}

\subsection{Dihedral cooperads}

	By dualizing Definitions~\ref{defimonads} and~\ref{defioperads}, one defines comonads of trees and the corresponding notions of cooperads. For more details, we refer the reader to~\cite[Section~$5.8.8$]{lodayvallettebook}. For instance, the comonad of trees is defined by the endofunctor~$\mathbb{T}^c:\mathsf{Bij}\text{-}\mathsf{Mod}\to \mathsf{Bij}\text{-}\mathsf{Mod}$ defined by 
	$$\mathbb{T}^c\mathcal{M}(S):=\bigoplus_{\tr\in\mathsf{Tree}(S)}\mathcal{M}(\tr)\ ,$$
	where we have set
	$$\mathcal{M}(\tr):=\bigotimes_{v\in V(\tr)}\mathcal{M}(E(v))\ .$$
	A cyclic cooperad consists of a~$\textsf{Bij}$-module~$\mathcal{C}$ along with decomposition morphisms
	$$\Delta_\tr:\mathcal{C}(S)\rightarrow\mathcal{C}(\tr)\ ,$$
	for any tree~$\tr\in\mathsf{Tree}(S)$, satisfying some coassociativity conditions. For the convenience of the reader, we make the definition explicit in the case of dihedral cooperads, switching from dihedral trees to polygon dissections (see Lemma~\ref{lem:Diss-Tree}).\\

	A~$\mathsf{Dih}$-module~$\mathcal{M}$ assigns to every dihedral set~$(S,\delta)$ an object~$\mathcal{M}(S,\delta)$, and to every dihedral bijection~$(S,\delta)\simeq (S',\delta')$ an isomorphism~$\mathcal{M}(S,\delta)\simeq\mathcal{M}(S',\delta')$. We introduce the notation, for a dissection~$\di\in\mathsf{Diss}(S,\delta)$:
	$$\mathcal{M}(\di):=\bigotimes_{p\in P(\di)}\mathcal{M}(E(p),\delta(p))\ .$$

\begin{defi}[Comonad of dissections]
The \emph{comonad of dissections}, denoted by~$\mathbb{DT}^c$, consists of the endofunctor~$\mathbb{DT}^c : \mathsf{Dih}\text{-}\mathsf{Mod}\to \mathsf{Dih}\text{-}\mathsf{Mod}$ defined by 
$$\mathbb{DT}^c\mathcal{M}(S,\delta):=\bigoplus_{\di\in\mathsf{Diss}(S,\delta)}\M(\di)\ .$$
Its law~$\mathbb{DT}^c\rightarrow\mathbb{DT}^c\circ\mathbb{DT}^c$ sends the direct summand indexed by a dissection~$\di$ to the direct summands indexed by all sub-dissections of~$\di$. The counit~$\mathbb{DT}^c\rightarrow 1$ is the projection on the direct summand indexed by empty dissections.
\end{defi}

\begin{defi}[Dihedral cooperad]
A \emph{dihedral cooperad} is a coalgebra over the comonad of dissections. 
\end{defi}

%\begin{rem}
%To be completely accurate, this notion should be called \emph{conilpotent} dihedral cooperad. The full notion of dihedral cooperad would be produced by the comonad made up of the {product}, and not the direct sum, over dihedral trees. We will ignore this subtlety in the sequel. 
%\end{rem}

The data of a dihedral cooperad is equivalent to a collection of \emph{decomposition morphisms} 
$$\Delta_{\di}\ : \ \mathcal{C}(S,\delta) \to \mathcal{C}(\di) \ ,~$$
for any dihedral tree~$\di \in \mathsf{Diss}(S, \delta)$, satisfying some coassociativity conditions. The first non-trivial  decomposition morphisms correspond to dissections with one chord; such decomposition morphisms are called \emph{infinitesimal} and their iterations can generate any decomposition morphism. 

\subsection{Cobar construction and cofree dihedral cooperads}
	In this subsection and in the next one, we assume that the underlying symmetric monoidal category~$\mathsf{A}$ consists of graded objects, like chain complexes for instance. We use the cohomological convention for cooperads. In this case, one can consider the \emph{desuspension}~$s^{-1}\calC$ of any dihedral module~$\calC$ defined by the formula~$s^{-1}\calC(S,\delta)^\bullet:=\calC(S,\delta)^{\bullet+1}$. (Alternatively, one can view the element~$s^{-1}$ as a dimension one element of the category~$\mathsf{A}$ concentrated in cohomological degree~$-1$. In this case, the desuspension coincide with the tensor product with the element~$s^{-1}$.)

\begin{defi}[Cobar construction]
The \emph{cobar construction} 
$\Omega\, \calC:=\big(\mathbb{DT}(s^{-1}\calC), \mathrm{d} \big)$
of a dihedral cooperad~$\calC$ is the free dihedral operad 
 generated by~$s^{-1}\calC$ equipped with the unique derivation~$\mathrm{d}$ which extends the infinitesimal decomposition morphisms of~$\calC$. The signs induced by the desuspension force the derivation~$\mathrm{d}$ to square to zero, which makes the cobar construction into a differential graded dihedral operad. 
\end{defi}

\begin{rem}
As usual~\cite[Section~$6.5.2$]{lodayvallettebook}, if the underlying dihedral module~$\calC$ carries an internal differential, one takes it into account in the definition of the cobar construction.  This will not be the case in the sequel.
\end{rem}

	The underlying cochain complex of the cobar construction looks like 
		\begin{equation*}
0\rightarrow s^{-1}\mathcal{C}(S, \delta) \rightarrow \bigoplus_{\di\in\mathsf{Diss}_1(S,\delta)} s^{-1}\mathcal{C}(\di) \rightarrow \bigoplus_{\di\in\mathsf{Diss}_2(S,\delta)} s^{-1}\mathcal{C}(\di) \rightarrow \cdots\ .\end{equation*}
	
One can read whether a dihedral cooperad is cofree on its cobar construction as follows. 
	
		\begin{prop}\label{propfreecobar}
		Let~$\mathcal{C}$ be a dihedral cooperad. The following assertions are equivalent:
		\begin{enumerate}[(i)]
		\item the dihedral cooperad~$\mathcal{C}$ is cofree;
		\item for every dihedral set~$(S, \delta)$, the cobar construction of~$\mathcal{C}$ induces a long exact sequence
		\begin{equation}\label{eqlongexseqcobar}
		s^{-1}\mathcal{C}(S, \delta) \rightarrow \bigoplus_{\di\in\mathsf{Diss}_1(S,\delta)} s^{-1}\mathcal{C}(\di) \rightarrow \bigoplus_{\di\in\mathsf{Diss}_2(S,\delta)} s^{-1}\mathcal{C}(\di) \rightarrow \cdots\ .
\end{equation}
		\end{enumerate}
		In such a situation, the space of cogenerators of~$\mathcal{C}$  is (non-canonically) isomorphic to the space of indecomposables 
		$$\mathcal{X}(S, \delta)=\mathrm{ker}\left( \mathcal{C}(S, \delta) \rightarrow \bigoplus_{\di\in\mathsf{Diss}_1(S,\delta)} \mathcal{C}(\di)  \right).$$
		More precisely, any choice of splitting for the inclusion of~$\mathsf{Dih}$-modules~$\mathcal{X}\hookrightarrow \mathcal{C}$ leads to an isomorphism
		$$\mathcal{C}\stackrel{\cong}{\longrightarrow} \mathbb{DT}^c(\mathcal{X})\ .$$
		\end{prop}
		
		\begin{proof}
		The long sequence (\ref{eqlongexseqcobar}) is exact if and only if the long sequence 
		\begin{equation}\label{eqlongexseqcobarbis}
		0\rightarrow s^{-1}\mathcal{X}(S, \delta) \rightarrow s^{-1}\mathcal{C}(S, \delta) \rightarrow \bigoplus_{\di\in\mathsf{Diss}_1(S,\delta)} s^{-1}\mathcal{C}(\di) \rightarrow \bigoplus_{\di\in\mathsf{Diss}_2(S,\delta)} s^{-1}\mathcal{C}(\di) \rightarrow \cdots \
		\end{equation}
		is exact. 
		
		$(i)\Rightarrow (ii)$:  Suppose that the dihedral cooperad~$\mathcal{C}\cong \mathbb{DT}^c(\mathcal{X})$ is cofree on a dihedral module~$\mathcal{X}$. Since the sequence (\ref{eqlongexseqcobarbis}) is  the analog of the bar-cobar resolution~\cite[Theorem 6.6.5]{lodayvallettebook} for the nilpotent dihedral operad~$s^{-1}\mathcal{X}$, one proves that this sequence is exact by the same kind of arguments.
		
		$(ii)\Rightarrow (i)$: Let us assume that the long sequence (\ref{eqlongexseqcobarbis}) is exact. We choose a splitting~$\mathcal{C}\twoheadrightarrow \mathcal{X}$ for the inclusions~$\mathcal{X}\hookrightarrow\mathcal{C}$ in the category of~$\mathsf{Dih}$-modules. This defines a morphism of dihedral cooperads~$\mathcal{C}\rightarrow \mathbb{DT}^c(\mathcal{X})$. Let us prove, by induction on the arity~$n\geq 3$ of a dihedral set~$(S,\delta)$,  that the morphism~$\mathcal{C}(S, \delta)\rightarrow \mathbb{DT}^c(\mathcal{X})(S, \delta)$ is an isomorphism. The case~$n=3$ is obvious and initiates the induction. Suppose that the property holds up to~$n-1$. We  prove that it holds for~$n$ as follows. The preceding point shows that the long  sequence (\ref{eqlongexseqcobarbis}) associated to the dihedral cooperad~$\mathbb{DT}^c(\mathcal{X})$ is exact. The induction hypothesis provides us with the following commutative diagram:
		$$
		\xymatrix@C=12pt{
		0 \ar[r]& s^{-1}\mathcal{X}(S, \delta)  \ar[r]\ar[d]^{\cong}& s^{-1}\mathcal{C}(S, \delta)  \ar[r]\ar[d]& \bigoplus_{\di\in\mathsf{Diss}_1(S,\delta)} s^{-1}\mathcal{C}(\di) \ar[r]\ar[d]^{\cong}&\bigoplus_{\di\in\mathsf{Diss}_2(S,\delta)} s^{-1}\mathcal{C}(\di) \ar[r]\ar[d]^{\cong}& \cdots \\
		0 \ar[r]& s^{-1}\mathcal{X}(S, \delta)  \ar[r]& s^{-1}\mathbb{DT}^c(\mathcal{X})(S,\delta) \ar[r]& \bigoplus_{\di\in\mathsf{Diss}_1(S,\delta)}  s^{-1}\mathbb{DT}^c(\mathcal{X})(\di)\ar[r]& \bigoplus_{\di\in\mathsf{Diss}_2(S,\delta)}  s^{-1}\mathbb{DT}^c(\mathcal{X})(\di)\ar[r]& \cdots 
		}$$
		where the rows are exact and nearly all the vertical maps are isomorphisms. A diagram chase (the~$5$-lemma) completes the proof. 
		\end{proof}
		
		\begin{prop}
		Let~$\mathcal{C}$ be a dihedral cooperad. Then the following statements are equivalent:
		\begin{enumerate}[(i)]
		\item The dihedral cooperad $\mathcal{C}$ is cofree.
		\item The nonsymmetric cyclic cooperad underlying $\mathcal{C}$ is cofree.
		\item The nonsymmetric cooperad underlying $\mathcal{C}$ is cofree.
		\end{enumerate}
		\end{prop}
		
		\begin{proof}
		The same proof shows that Proposition~\ref{propfreecobar} is valid in the category of nonsymmetric cyclic cooperads (resp. nonsymmetric cooperads), replacing the dihedral cobar construction by the nonsymmetric cyclic cobar construction (resp. the nonsymmetric cobar construction). By Remark~\ref{remfreeoperad}, these three cobar constructions have the same underlying nonsymmetric operad, which is the nonsymmetric cobar construction of the nonsymmetric cooperad underlying~$\mathcal{C}$. In particular, they have the same underlying chain complex, and the claim follows.
		\end{proof}
		
		\subsection{The coradical filtration and a freeness criterion}

To understand the behavior of a dihedral cooperad with respect to the freeness property, one can consider its coradical filtration. This is the direct generalisation of the same notion on the level of coalgebras~\cite[Appendix~B]{quillenrational} and on the level of cooperads~\cite[Section~$5.8.4$]{lodayvallettebook}. 
		
		\begin{defi}[Coradical filtration]
		Let~$\calC$ be a dihedral cooperad. The \emph{coradical filtration}, defined by 
		~$$F_k \calC  (S,\delta) :=\bigcap_{\di \in \mathsf{Diss}_{k+1}(S,\delta)} \ker(\Delta_\di)\ ,~$$ 
		 for~$k\geqslant 0$, is an increasing filtration of the~$\mathsf{Dih}$-module~$\calC$: 
		~$$0=F_{-1} \calC\subset F_{0} \calC \subset F_{1} \calC \subset \cdots \subset \calC \ .$$
		\end{defi}
		
		The next proposition gives a way to recognize coradical filtrations of cofree dihedral cooperads.  Let us make the following convention: if we are given an increasing filtration~$\cdots \subset R_{k-1} \mathcal{C}\subset R_k \mathcal{C}\subset \cdots$ of a~$\mathsf{Dih}$-module~$\mathcal{C}$, then we extend this filtration, in the natural way, to all objects~$\mathcal{C}(\di)$, for a dissection~$\di$ as follows. If  the dissection~$\di$ dissects~$(S,\delta)$ into polygons~$p_0, p_1, \ldots, p_k$, then we set
		$$R_r \mathcal{C} (\di):=\sum_{i_0+\cdots+i_k=r}R_{i_0} \mathcal{C} (p_0)\otimes\cdots\otimes R_{i_k} \mathcal{C} (p_k)\ .$$
	
		\begin{prop}\label{propfreefiltration}
		Let~$\calC$ be a dihedral cooperad. Assume that the underlying~$\mathsf{Dih}$-module is equipped with an increasing filtration 
		$$0=R_{-1} \mathcal{C} \subset R_0 \mathcal{C} \subset R_1 \mathcal{C}\subset \cdots \subset \calC$$
		which is finite in every arity $n$ and such that the following properties are satisfied:
		\begin{enumerate}
		\item[(a)] for every dissection~$\di\in \mathsf{Diss}_k(S,\delta)$  of cardinality~$k$ and every integer~$r$, the decomposition map~$\Delta_\di$ sends~$R_r \mathcal{C} (S,\delta)$ to~$R_{r-k} \mathcal{C} (S,\delta)$;
		\item[(b)] for every integer~$r$, the iterated decomposition map
		\begin{equation}\label{eqcocompiso}
		\mathrm{gr}_r^R\mathcal{C} (S,\delta) \xrightarrow{\bigoplus\Delta_\di}\bigoplus_{\di\in \mathsf{Diss}_r(S,\delta)}R_0 \mathcal{C} (\di)
		\end{equation}
		is an isomorphism.
		\end{enumerate}
 Then the dihedral cooperad~$\mathcal{C}$ is cofree and the filtration~$R$ is its coradical filtration. More precisely, any choice of splitting of the inclusion~$R_0 \mathcal{C} \hookrightarrow \mathcal{C}$ induces an isomorphism
		$$\mathcal{C}\stackrel{\cong}{\longrightarrow} \mathbb{DT}^c(R_0 \mathcal{C})\ .$$
		\end{prop}
		
		\begin{proof}
		Let us choose a splitting~$\theta :\mathcal{C} \twoheadrightarrow R_0 \mathcal{C}$ for the inclusions~$R_0 \mathcal{C}\hookrightarrow\mathcal{C}$ in the category of $\mathsf{Dih}$-modules. By the universal property of the cofree dihedral cooperads, this induces a morphism of dihedral cooperads~$\Theta : \mathcal{C}\rightarrow \mathbb{DT}^c(R_0 \mathcal{C})$. 
		The coradical filtration on the cofree dihedral cooperad~$\mathbb{DT}^c(R_0 \mathcal{C})$ is given by
		$$F_k \mathbb{DT}^c(R_0\, \mathcal{C}) (S,\delta)=\bigoplus_{\substack{r\leq k \\ \di \in \mathsf{Diss}_r(S,\delta)}} R_0 \mathcal{C}(\di).$$
		For any dissection~$\di\in \mathsf{Diss}_k(S,\delta)$ and~$k>r$, the first assumption implies that we have~$\Delta_\di(R_r \mathcal{C}(S,\delta))=0$. Therefore the morphism~$\Theta$ is compatible with the filtrations~$R$ and it induces a morphism of graded dihedral modules
		$$\mathrm{gr}^R_r  \Theta\ :\  \mathrm{gr}^R_r  \mathcal{C}(S,\delta) \rightarrow \mathrm{gr}^R_r \mathbb{DT}^c(R_0\, \mathcal{C})=\bigoplus_{\di \in \mathsf{Diss}_r(S,\delta)}R_0 \mathcal{C}(\di)\ ,$$
		which is nothing but the iterated decomposition map~(\ref{eqcocompiso}). So it is an isomorphism by the second assumption. Finally, the morphism of dihedral cooperads~$\Theta$ is an isomorphism and the proposition is proved.
		\end{proof}

\section{Moduli spaces of genus zero curves and the cyclic gravity operad}

	In this section, we begin by recalling the definitions of the moduli space of genus zero curves with marked points and its Deligne--Mumford--Knudsen compactification. We recall the definition of residues along normal crossing divisors in the context of mixed Hodge theory. This produces the cyclic gravity operad structure on the cohomology of the moduli spaces of curves.
	
	\subsection{Normal crossing divisors and stratifications}\label{parncdstrat}
	
		We introduce some vocabulary and notations on normal crossing divisors and the stratifications that they induce on complex algebraic varieties.
	
		\subsubsection{The local setting}\label{parncdstratlocal}
		
		Let~$\overline{X}$ be a small neighbourhood of~$0$ in~$\C^n$ and let us define a divisor~$\partial\overline{X}=\{z_1\cdots z_r=0\}$ in~$\overline{X}$, for some fixed integer~$r$. Its irreducible components are the (intersections with~$\overline{X}$ of the) coordinate hyperplanes~$\{z_i=0\}$ for~$i=1,\ldots,r$. This induces a stratification
		\begin{equation}\label{eqstratlocal}
		\overline{X}=\bigsqcup_{I\subset\{1,\ldots,r\}} X(I)\ ,
		\end{equation}
		where~$X(I)$ is the locally closed subset of~$\overline{X}$ defined by the conditions:~$z_i=0$ for~$i\in I$ and~$z_i\neq 0$ for~$i\in \{1,\ldots,r\}\setminus I$. 
		Notice that 
		$$I\subset I' \;\Leftrightarrow \; \overline{X}(I)\supset \overline{X}(I').$$
		The  codimension of~$X(I)$ is equal to the cardinality of~$I$, and its closure~$\overline{X}(I)$ is defined by the vanishing of the coordinates~$z_i$,~$i\in I$. 
		In other words, the closure~$\overline{X}(I)$ is the union of the strata~$X(I')$, for~$I'\supset I$:
		$$ \overline{X}(I) = \bigsqcup_{I'\supset I} X(I')\ .$$
		
		For a given set~$I\subset\{1,\ldots,r\}$, the complement~$\partial\overline{X}(I):=\overline{X}(I)\setminus X(I)$ is defined by the equation~$\prod_{i\in\{1,\ldots,r\}\setminus I}z_i=0$.

		\subsubsection{The global setting}\label{parncdstratglobal}
		
		Let~$\overline{X}$ be a smooth (not necessarily compact) complex algebraic variety and let~$\partial\overline{X}$ be a normal crossing divisor inside~$\overline{X}$. This means that around every point of~$\overline{X}$, there is a system of coordinates~$(z_1,\ldots,z_n)$, where~$n$ is the complex dimension of~$\overline{X}$, such that~$\partial\overline{X}$ is defined by an equation of the form~$z_1\cdots z_r=0$ for some integer~$r$ that depends on the point.
		
		This induces a global stratification:
		\begin{equation}\label{eqstratX}
		\overline{X}=\bigsqcup_{\mathfrak{s}\in\mathsf{Strat}}X(\mathfrak{s})\ ,
		\end{equation}
		that is constructed as (\ref{eqstratlocal}) in every local chart.
 For every~$\mathfrak{s}$ in the indexing set~$\mathsf{Strat}$, the stratum~$X(\mathfrak{s})$ is a connected locally closed subset of~$\overline{X}$. Let~$\overline{X}(\mathfrak{s})$ denote its closure. The indexing set~$\mathsf{Strat}$ for the strata is actually endowed with a poset structure defined by
		$$\mathfrak{s}\leq\mathfrak{s'} \;\Leftrightarrow\; \overline{X}(\mathfrak{s})\supset \overline{X}(\mathfrak{s'})$$
		In other words, the closure~$\overline{X}(\mathfrak{s})$ of~$X(\mathfrak{s})$ is the union of the strata~$X(\mathfrak{s'})$, for~$\mathfrak{s'}\geq \mathfrak{s}$:
		$$ \overline{X}(\mathfrak{s}) = \bigsqcup_{\mathfrak{s'}\geq \mathfrak{s}} X(\mathfrak{s'})\ .$$
		
		For an integer~$k$, we write~$\mathsf{Strat}_k$ for the indexing set of strata of codimension~$k$, making~$\mathsf{Strat}$ into a graded poset. The set~$\mathsf{Strat}_0$ only has one element corresponding to the open stratum~$X=\overline{X}\setminus\partial\overline{X}$. The closures~$\overline{X}(\mathfrak{s})$, for~$\mathfrak{s}\in \mathsf{Strat}_1$, are the irreducible components of the normal crossing divisor~$\partial\overline{X}$.
			
		For a given stratum~$X(\mathfrak{s})$, the complement~$\partial\overline{X}(\mathfrak{s}):=\overline{X}(\mathfrak{s})\setminus X(\mathfrak{s})$ is a normal crossing divisor inside~$\overline{X}(\mathfrak{s})$. 
						
	\subsection{The moduli spaces~$\M_{0,S}$ and~$\Mbar_{0,S}$}
	
		We introduce the moduli spaces of genus zero curves~$\M_{0,S}$ and~$\Mbar_{0,S}$. We refer the reader to~\cite{knudsen,keel,getzlermodulispacesgenuszero,goncharovmanin} for more details. 

		\subsubsection{The open moduli spaces~$\M_{0,S}$}	
	
		Let~$S$ be a finite set of cardinality~$n\geq 3$. The \emph{moduli space of genus zero curves with~$S$-marked points} is the quotient of the configuration space of points labeled by~$S$ on the Riemann sphere~$\mathbb{P}^1(\C)$ by the automorphisms of~$\mathbb{P}^1(\C)$. It is denoted by
	$$\M_{0,S}:=\{(z_s)_{s\in S} \in \mathbb{P}^1(\C)^S \; | \; \forall s\neq s' \, , \, z_s\neq z_{s'} \} \, / \, \mathrm{PGL}_2(\C)\ ,$$ 
	where an element~$g\in\mathrm{PGL}_2(\C)$ acts diagonally by~$g.(z_s)_{s\in S}=(g.z_s)_{s\in S}$\ .
	
		Every bijection~$S\simeq S'$ induces an isomorphism~$\M_{0,S}\simeq\M_{0,S'}$. If~$S=\{1,\ldots,n\}$ then~$\M_{0,S}$ is simply denoted by~$\M_{0,n}$.\\
 	
		The action of~$\mathrm{PGL}_2(\C)$ on~$\mathbb{P}^1(\C)$ is strictly tritransitive: for every triple~$(a,b,c)$ of pairwise distinct points on~$\mathbb{P}^1(\C)$, there exists a unique element~$g\in\mathrm{PGL}_2(\C)$ such that~$(g.a,g.b,g.c)=(\infty,0,1)$ . By fixing an identification
		$$(z_1,\ldots,z_n)=(\infty,0,t_1,\ldots,t_{n-3},1)\ ,$$
		we can thus get rid of the quotient by~$\mathrm{PGL}_2(\C)$ and obtain an isomorphism
		\begin{equation}\label{eqMcomplementarrangement}
		\M_{0,n}\simeq \{(t_1,\ldots,t_{n-3})\in \C^{n-3}\;|\; \forall i \, ,\, t_i\neq 0,1 \; ; \;  \forall i\neq j \,,\; t_i\neq t_j \}\ .
		\end{equation}
		This description makes it clear that~$\M_{0,S}$ is a smooth and affine complex algebraic variety of dimension~$n-3$.
		
		\subsubsection{The compactified moduli spaces~$\Mbar_{0,S}$}		
		
		Let~$S$ be a finite set of cardinality~$n\geq 3$, and let 
		$$\M_{0,S}\subset\Mbar_{0,S}$$
		be the Deligne--Mumford--Knudsen compactification of~$\M_{0,S}$. Every bijection~$S \simeq S'$ induces an isomorphism~$\Mbar_{0,S}\simeq\Mbar_{0,S'}$. If~$S=\{1,\ldots,n\}$ then~$\Mbar_{0,S}$ is simply denoted by~$\Mbar_{0,n}$.\\
		
		The compactified moduli space~$\Mbar_{0,S}$ is a smooth projective complex algebraic variety, and the complement~$\partial\Mbar_{0,S}:=\Mbar_{0,S}\setminus \M_{0,S}$ is a simple normal crossing divisor. The corresponding stratification (\ref{eqstratX}) is indexed by the graded poset of~$S$-trees:
		\begin{equation}\label{eqstratMbar}
		\Mbar_{0,S}=\bigsqcup_{\tr\in \mathsf{Tree}(S)} \M(\tr)\ .
		\end{equation}
		The codimension of a stratum~$\M(\tr)$ is equal to the number of internal edges of the tree~$\tr$. If we denote by~$\Mbar(\tr)$ the closure of a stratum~$\M(\tr)$ in~$\Mbar_{0,S}$, then we have		
		$$\Mbar(\tr)\supset \Mbar(\tr') \;\Leftrightarrow \; \tr\leq \tr'\ ,$$
		where the order~$\leq$ on trees is the one defined in Definition~\ref{def-Tree}. The closure~$\Mbar(\tr)$ is thus the union of the strata~$\M(\tr')$, for~$\tr'\geq \tr$.
		
		For a tree~$\tr\in\mathsf{Tree}(S)$, we have compatible product decompositions
		\begin{equation}\label{eqproductdecMMbar}
		\M(\tr)\cong \prod_{v\in V(\tr)} \M_{0,E(v)} \;\;\textnormal{ and }\;\; \Mbar(\tr)\cong \prod_{v\in V(\tr)} \Mbar_{0,E(v)}\ .
		\end{equation}
		
		The stratum corresponding to the corolla is the open stratum~$\M_{0,S}$. For~$\tr\in\mathsf{Tree}_1(S)$ a tree with only one internal edge, we get a divisor
		$$\Mbar(\tr)\cong \Mbar_{0,E_0}\times\Mbar_{0,E_1}$$
		inside~$\Mbar_{0,S}$. These divisors are the irreducible components of~$\partial\Mbar_{0,S}$.  
	
		\begin{ex}\leavevmode
		\begin{enumerate} 
		\item We have~$\M_{0,3}=\Mbar_{0,3}=\{*\}$.
		\item If we write~$\M_{0,4}=\mathbb{P}^1(\C)\setminus\{\infty,0,1\}$ as in (\ref{eqMcomplementarrangement}), then we have~$\Mbar_{0,4}=\mathbb{P}^1(\C)$. The divisor at infinity~$\partial\Mbar_{0,4}=\{\infty,0,1\}$ has three irreducible components, all isomorphic to a product~$\Mbar_{0,3}\times\Mbar_{0,3}$, indexed by the three~$4$-trees with one internal edge.
		\item If we write~$\M_{0,5}=(\mathbb{P}^1(\C)\setminus\{\infty,0,1\})^2\setminus\{t_1=t_2\}$ as in (\ref{eqMcomplementarrangement}), then~$\Mbar_{0,5}$ can be realized as the blow-up of~$\mathbb{P}^1(\C)^2$ along the three points~$(0,0)$,~$(1,1)$,~$(\infty,\infty)$, see Figure~\ref{figureMzerocinq}. The divisor at infinity~$\partial\Mbar_{0,5}$ has ten irreducible components: the three exceptional divisors and the strict transforms of the lines~$t_1=0,1,\infty$,~$t_2=0,1,\infty$,~$t_1=t_2$. They are all isomorphic to a product~$\Mbar_{0,3}\times\Mbar_{0,4}$ and are indexed by the ten~$5$-trees with one internal edge. The fifteen different intersection points of these components are indexed by the fifteen~$5$-trees with two internal edges.
		\end{enumerate}
		\end{ex}
		
		\begin{figure}[h!!]
		\def\svgwidth{.3\textwidth}
		\begin{center}
\begin{tikzpicture}[scale=0.27]
\draw[thick] (-11.5, 10) -- (8,10);
\draw[thick] (-11.5, 0) -- (-2,0);
\draw[thick] (2, 0) -- (11.5,0);
\draw[thick] (-8, -10) -- (11.5,-10);

\draw[thick] (-10, 11.5) -- (-10,-8);
\draw[thick] (0, 11) -- (0,2);
\draw[thick] (0, -2) -- (0,-11);
\draw[thick] (10, 8) -- (10,-11);

\draw[thick] (1.5,1.5) -- (8.5,8.5);
\draw[thick] (-1.5,-1.5) -- (-8.5,-8.5);

\draw [thick] (150:3) arc [radius=3, start angle=150, end angle= 300];
\draw [thick] (7.42,11.5) arc [radius=3, start angle=150, end angle= 300];
\draw [thick] (-7.42, -11.5) arc [radius=3, start angle=-30, end angle= 120];
\end{tikzpicture}
		\end{center}
		\caption{The combinatorial structure of~$\Mbar_{0,5}$.}\label{figureMzerocinq}
		\end{figure}
	
	\subsection{The category of mixed Hodge structures}
	
		We recall some useful facts on the category of mixed Hodge structures. The main references are the original articles by Deligne~\cite{delignehodge1, delignehodge2, delignehodge3} and the book~\cite{peterssteenbrink}.
	
		\begin{defi}[pure Hodge structures]\leavevmode
		A \textit{pure Hodge structure} of weight~$w$ is the data of 
		\begin{enumerate}[--]
		\item a finite-dimensional~$\Q$-vector space~$H$\ ;
		\item a finite decreasing filtration, the \textit{Hodge filtration}, ~$F^\bullet H_\C$ of the complexification~$H_\C:=H\otimes_\Q\C$\ ,
		\end{enumerate}
		such that for every integer~$p$, we have
		$$H_\C=F^pH_\C \oplus \overline{F^{w-p+1}H_\C}\ .$$
		A morphism of pure Hodge structures is a morphism of~$\Q$-vector spaces that is compatible with the Hodge filtration.
		\end{defi}
				
		\begin{defi}[mixed Hodge structures]
		 A \textit{mixed Hodge structure} is the data of 
		\begin{enumerate}[--]
		\item a finite-dimensional~$\Q$-vector space~$H$;
		\item a finite increasing filtration, the \textit{weight filtration}, ~$W_{\bullet}H$ of~$H$\ ;
		\item a finite decreasing filtration, the \textit{Hodge filtration}, ~$F^\bullet H_\C$ of the complexification~$H_\C$\ ,
		\end{enumerate}
		such that for every integer~$w$, the Hodge filtration induces a pure Hodge structure of weight~$w$ on~$\mathrm{gr}^W_wH:=W_w H / W_{w-1}H\ .$ 
		A morphism of mixed Hodge structures is a morphism of~$\Q$-vector spaces that is compatible with the weight and Hodge filtrations.
		\end{defi}
		
		A pure Hodge structure of weight~$w$ is thus nothing but a mixed Hodge structure whose weight filtration is concentrated in weight~$w$. 
		
		A very important remark is that morphisms of mixed Hodge structures are strictly compatible with the weight and Hodge filtrations. This implies that mixed Hodge structures form an abelian category. One easily defines on it a compatible structure of a symmetric monoidal category.
		
		Another consequence of this strictness property is the following lemma, used in practice to prove degeneration of spectral sequences, like in Proposition~\ref{lemdegen}.
		
		\begin{lem}\label{lemappendixzero}
		Let~$f:H\rightarrow H'$ be a morphism of mixed Hodge structures. If~$H$ is pure of weight~$w$ and~$H'$ is pure of weight~$w'$ with~$w\neq w'$, then~$f=0$.
		\end{lem}
	
		The \textit{pure Tate structure of weight~$2k$}, denoted by~$\Q(-k)$, is the only pure Hodge structure of weight~$2k$ and dimension~$1$; its Hodge filtration is concentrated in degree~$k$. They satisfy~$\Q(-k)\otimes\Q(-l)\cong\Q(-k-l)$ and~$\Q(-k)^\vee\cong\Q(k)$. A mixed Hodge structure is said to be \textit{pure Tate of weight~$2k$} if it is isomorphic to a direct sum~$\Q(-k)^{\oplus d}$ for a certain integer~$d$.
		
		If~$H$ is a mixed Hodge structure and~$k$ is an integer, we denote by~$H(-k)$ the \textit{Tate twist} of~$H$ consisting in shifting the weight filtration by~$2k$ and the Hodge filtration by~$k$. It is equal to the tensor product of~$H$ by~$\Q(-k)$.\\
		
		The importance of mixed Hodge structures in the study of the topology of complex algebraic varieties is explained by the following fundamental theorem of Deligne.
		
		\begin{thm}\cite[Proposition 8.2.2]{delignehodge3}
		Let~$X$ be a complex algebraic variety. For every integer~$k$, the cohomology group~$H^k(X)$ is endowed with a functorial mixed Hodge structure. 
		\end{thm}	
		
	\subsection{Logarithmic forms and residues}
	
		We recall the notion of logarithmic form along a normal crossing divisor and that of a residue. We refer the reader to~\cite[3.1]{delignehodge2} for more details.
	
		\subsubsection{The local setting}
		
			We work in the local setting of Section~\ref{parncdstratlocal}. We say that a meromorphic differential form on~$\overline{X}$ has \textit{logarithmic poles} along~$\partial\overline{X}$, or that it is a \textit{logarithmic form} on~$(\overline{X},\partial\overline{X})$, if it can be written as a linear combination of forms of the type
			$$\dfrac{dz_{i_1}}{z_{i_1}}\wedge\cdots\wedge\dfrac{dz_{i_s}}{z_{i_s}}\wedge\eta\ ,$$
			with~$1\leq i_1<\cdots<i_s\leq r$ and where~$\eta$ a holomorphic form on~$\overline{X}$. Logarithmic forms are closed under the exterior derivative on forms.
			
			Any logarithmic form on~$(\overline{X},\partial\overline{X})$ can be written as
			$$\omega=\dfrac{dz_1}{z_1}\wedge\alpha+\beta\ ,$$
			where~$\alpha$ and~$\beta$ are forms with logarithmic poles along~$\{z_2\cdots z_r=0\}$. We define the \textit{residue} of~$\omega$ on~$\overline{X}(1)=\{z_1=0\}$ to be the restriction
			\begin{equation}\label{eqreslocal}
			\mathrm{Res}(\omega):=2\pi i\, \alpha_{|\overline{X}(1)}\ .
			\end{equation}
			It is a well-defined logarithmic form on~$(\overline{X}(1),\partial\overline{X}(1))$. The residue operation lowers the degree of the forms by~$1$ and anticommutes with the exterior derivative:~$d\circ\mathrm{Res}+\mathrm{Res}\circ d=0$.
			
			More generally, for sets~$I\subset I' \subset\{1,\ldots,r\}$ with~$|I'|=|I|+1$, we get residue operations~$\mathrm{Res}^I_{I'}$ from logarithmic forms on~$(\overline{X}(I),\partial\overline{X}(I))$ to logarithmic forms on~$(\overline{X}(I'),\partial\overline{X}(I'))$.
		
		\subsubsection{The global setting}
		
			We work in the global setting of Section~\ref{parncdstratglobal}. By gluing together the local definitions of the previous paragraph, one defines on each closure~$\overline{X}(\mathfrak{s})$ a complex of sheaves of logarithmic forms on~$(\overline{X}(\mathfrak{s}),\partial\overline{X}(\mathfrak{s}))$:
			$$\Omega^\bullet_{\overline{X}(\mathfrak{s})}(\log\partial\overline{X}(\mathfrak{s}))\ .$$

			If~$j_{\mathfrak{s}}:X(\mathfrak{s})\hookrightarrow \overline{X}(\mathfrak{s})$ denotes the natural open immersion, we have a quasi-isomorphism
			$(j_{\mathfrak{s}})_*\C_{X(\mathfrak{s})}\isom\Omega^\bullet_{\overline{X}(\mathfrak{s})}(\log\partial\overline{X}(\mathfrak{s}))$, which induces isomorphisms between cohomology groups:
		\begin{equation}\label{eqqis}
		H^k(X(\mathfrak{s}),\C) \, \cong\,  \mathbb{H}^k(\overline{X}(\mathfrak{s}),\Omega^{\bullet}_{\overline{X}(\mathfrak{s})}(\log\partial\overline{X}(\mathfrak{s})))\ .
		\end{equation}
		
		For elements~$\mathfrak{s}\leq\mathfrak{s}'$ in~$\mathsf{Strat}$ with~$|\mathfrak{s}'|=|\mathfrak{s}|+1$, we denote the corresponding closed immersion by~$i_{\mathfrak{s}'}^{\mathfrak{s}}:\overline{X}(\mathfrak{s}')\hookrightarrow \overline{X}(\mathfrak{s})$. By applying the local construction of the previous paragraph in every local chart, we get a residue morphism
		
		\begin{equation}\label{eqresOmega}
		\mathrm{Res}_{\mathfrak{s}'}^{\mathfrak{s}}:\Omega^\bullet_{\overline{X}(\mathfrak{s})}(\log\partial \overline{X}(\mathfrak{s}))\rightarrow (i_{\mathfrak{s}'}^{\mathfrak{s}})_*\Omega^{\bullet-1}_{\overline{X}(\mathfrak{s}')}(\log \partial\overline{X}(\mathfrak{s}'))\ , 
		\end{equation}
		which anticommutes with the exterior derivative on forms. In view of (\ref{eqqis}), this induces a residue morphism between cohomology groups:
		$$\mathrm{Res}_{\mathfrak{s}'}^{\mathfrak{s}}: H^\bullet\big(X(\mathfrak{s}),\C\big)\rightarrow H^{\bullet-1}(X(\mathfrak{s'}),\C)\ .$$
		
		This residue morphism is actually defined over~$\Q$ and it is compatible with the mixed Hodge structures if we add the right Tate twist, giving rise to residue morphisms
		\begin{equation}\label{eqresHodge}
		\mathrm{Res}_{\mathfrak{s}'}^{\mathfrak{s}}: H^\bullet(X(\mathfrak{s}))\rightarrow H^{\bullet-1}(X(\mathfrak{s}'))(-1)\ .
		\end{equation}

	\subsection{The cyclic gravity cooperad}\label{parcyclicGrav}
	
		Following Getzler, we use the residue morphisms of the previous paragraph to define the cyclic gravity cooperad in the category of graded mixed Hodge structures. Let~$S$ be a finite set of cardinality~$n\geq 3$, and let us choose an~$S$-tree~$\tr\in\mathsf{Tree}_1(S)$  with one internal edge. Let us denote by~$v_0$ and~$v_1$ its two vertices and by~$E_0:=E(v_0)$ and~$E_1:=E(v_1)$ the corresponding sets of adjacent edges. The stratum indexed by~$\tr$ in the moduli space~$\M_{0,S}$ is
		$$\M(\tr)\cong \M_{0,E_0}\times\M_{0,E_1}\ .$$
		For integers~$a$ and~$b$, we thus get residue morphisms
		\begin{equation}\label{eqRes12}
		\Delta_{\tr}:H^{a+b-1}(\M_{0,S})(-1)\rightarrow H^{a-1}(\M_{0,E_0})(-1) \otimes H^{b-1}(\M_{0,E_1})(-1)\ .
		\end{equation}
		They are obtained from (\ref{eqresHodge}) by using the K\"{u}nneth formula, adding a Tate twist~$(-1)$ and multiplying by the Koszul sign~$(-1)^{a-1}$, which reflects the cohomological degree shift. Let us define the~$\textsf{Bij}$-module~$\mathcal{C}$ in the category of graded mixed Hodge structures by 
		$$\mathcal{C}(S):=H^{\bullet-1}(\M_{0,S})(-1)\ .$$
	Associated to any set~$V$, one considers the one dimensional vector space~$\det (V):=\bigwedge_{v\in V} \,\mathbb{Q}v$. 		The signed residue morphisms (\ref{eqRes12}) are not quite the decomposition morphisms of a cyclic cooperad. Instead, they give rise to decomposition morphisms
		\begin{equation}\label{eqdeltadet}
		\Delta_{\mathfrak{t}} : \mathcal{C}(S) \rightarrow \det(V(\mathfrak{t})) \otimes \mathcal{C}(\mathfrak{t})\ ,
		\end{equation}
		for any~$S$-tree~$\tr\in\mathsf{Tree}(S)$, that satisfy analogs of the axioms a cyclic cooperad, but with different signs. Such an algebraic structure on~$\mathcal{C}$ is actually called an \emph{anti-cyclic cooperad}, see~\cite[2.10]{getzlerkapranov}. Note that in (\ref{eqRes12}) the choice of an ordering~$V(\tr)=\{v_0,v_1\}$ gives a trivialization~$\det(V(\mathfrak{t}))\simeq \Q$ of the determinant. The following definition was introduced by Getzler~\cite{getzlertopologicalgravity,getzlermodulispacesgenuszero}.
		
		\begin{defi}[Cyclic gravity cooperad] The \emph{cyclic gravity cooperad} is the cyclic suspension~\cite[2.10]{getzlerkapranov} of the anti-cyclic cooperad~$\mathcal{C}$: 		
		$$\Grav(S):=\det(S)\otimes H^{\bullet+n-3}(\M_{0,S})(-1)\ ,~$$
		for any finite set~$S$  of cardinality~$n\geq 3$.
		It forms a cyclic cooperad in the category of graded mixed Hodge structures, which is concentrated in non-positive cohomological degree~$-(n-3)\leq\bullet\leq 0$. The decomposition morphisms
		$$\Delta_{\mathfrak{t}}:\Grav(S)\rightarrow \Grav(\mathfrak{t})$$
		for the cyclic gravity cooperad are given by signed residues.
		\end{defi}
		
		Getzler showed~\cite[Theorem 4.5]{getzlertopologicalgravity} that the cyclic gravity operad, linear dual to the cyclic gravity cooperad, is generated by one element in each cyclic arity~$n\geq 3$, and he also gave a presentation for the operadic ideal of relations. More specifically, the generator in cyclic arity~$n$ is the natural generator of the space~$H_0(\M_{0,n})(1)$, which lies in homological degree~$-(n-3)$, and the relations are generalizations of the Jacobi identity for Lie algebras. In particular, the generator of cyclic arity~$3$ satisfies the Jacobi identity, and one gets the following theorem.
		
		\begin{thm}\label{thmlie}\cite[3.8]{getzlermodulispacesgenuszero}
		The degree zero sub-operad of the cyclic gravity operad is isomorphic to the cyclic Lie operad. In particular, we get an isomorphism of~$\mathsf{Bij}$-modules
		$$\mathcal{L}\textit{ie}\,(S) \cong \det(S)\otimes H_{n-3}(\M_{0,S})(1)\ .$$
		\end{thm}
	
\section{Brown's moduli spaces and the dihedral gravity cooperad}

	In this section, we introduce Brown's moduli spaces as a partial compactification of the moduli spaces of genus zero curves. Forgetting many of the symmetries of the gravity operad, one obtains the dihedral gravity operad. We conclude with the proof of the equivalence between the purity of the mixed Hodge structure on the cohomology of Brown's moduli spaces and the cofreeness of the dihedral gravity cooperad.

	\subsection{Brown's moduli spaces~$\Mdelta_{0,S}$}
	
		Let~$S$ be a finite set of cardinality~$n\geq 3$ and let~$\delta$ be a dihedral structure on~$S$. Brown defined~\cite[\S 2]{brownPhD} a space~$\Mdelta_{0,S}$ that fits between the moduli space~$\M_{0,S}$ and its compactification~$\Mbar_{0,S}$ with open immersions:
		$$\M_{0,S}\subset \Mdelta_{0,S} \subset \Mbar_{0,S}\ .$$
		
		Recall that~$\mathsf{DTree}(S,\delta)\subset\mathsf{Tree}(S)$ denotes the set of~$S$-trees that have a dihedral embedding compatible with~$\delta$.
		
		\begin{defi}[Brown's moduli space~$\Mdelta_{0,S}$]
		Brown's moduli space~$\Mdelta_{0,S}$ is the subspace of~$\Mbar_{0,S}$ defined as the union of strata indexed by the trees underlying dihedral trees:
		$$\Mdelta_{0,S}:=\bigsqcup_{\tr\in \mathsf{DTree}(S,\delta)} \M(\tr)\ .$$
		\end{defi}
		
		For~$\tr$ and~$\tr'$ two~$S$-trees such that~$\tr\leq\tr'$, we have~$\tr'\in\mathsf{DTree}(S,\delta) \Rightarrow \tr\in\mathsf{DTree}(S,\delta)$; thus, Brown's moduli space~$\Mdelta_{0,S}$ is an open subvariety of~$\Mbar_{0,S}$. In other words, it is the complement in~$\Mbar_{0,S}$ of the union of the closed subvarieties~$\Mbar(\tr)$, for~$\tr\in \mathsf{Tree}(S)\setminus  \mathsf{DTree}(S,\delta)$; in this description, it is actually enough to delete the divisors~$\Mbar(\tr)$, for trees~$\tr$  with one internal edge.
		
		Every dihedral bijection~$(S,\delta)\simeq (S',\delta')$ induces an isomorphism~$\M^{\delta}_{0,S}\simeq\M^{\delta'}_{0,S'}$. If we consider~$S=\{1,\ldots,n\}$ with its standard dihedral structure ~$\delta$, then~$\Mdelta_{0,S}$ is simply denoted by~$\Mdelta_{0,n}$.
		
		\begin{thm}\label{thmbrownaffine}\cite[Theorem 2.21]{brownPhD}
		Brown's moduli space~$\Mdelta_{0,S}$ is a smooth and affine complex algebraic variety, and the complement~$\partial\Mdelta_{0,S}:=\Mdelta_{0,S}\setminus\M_{0,S}$ is a normal crossing divisor.
		\end{thm}
		
		With our definition of Brown's moduli spaces, the only non-trivial statement in the above theorem is the fact that~$\Mdelta_{0,S}$ is affine. Brown's original definition is via an explicit presentation of the ring of functions of~$\Mdelta_{0,S}$. The equivalence of the two definitions can be found in~\cite[\S 2.6]{brownPhD}.
		
		\subsection{The dihedral gravity cooperad}
		
		\begin{defi}[The dihedral gravity operad]
		The \textit{dihedral gravity cooperad}, still denoted by~$\Grav$, is the dihedral cooperad in the category of graded mixed Hodge structures underlying the cyclic gravity cooperad. In other words, it is obtained by applying the forgetful functor $\mathsf{Cyc}\text{-}\mathsf{Op}\rightarrow\mathsf{Dih}\text{-}\mathsf{Op}$ of Section~\ref{sec:DiOp} to the cyclic gravity operad. Recall that its underling dihedral module is given by 
		$$\Grav(S, \delta):=\det(S)\otimes H^{\bullet+n-3}(\M_{0,S})(-1)\ .~$$
		\end{defi}
		
		 For the convenience of the reader and for future use, we restate its definition in the dihedral setting by using the bijection between graded posets of Lemma~\ref{lem:Diss-Tree}:
		$$\mathsf{DTree}(S,\delta)\cong\mathsf{Diss}(S,\delta) \; , \; \tr \leftrightarrow \di \ .$$
		We may then write
		\begin{equation}\label{eqstratdelta}
		\Mdelta_{0,S}=\bigsqcup_{\di\in\mathsf{Diss}(S,\delta)} \M(\di)\ .
		\end{equation}
		The codimension of a stratum~$\M(\di)$ is the number of chords in the dissection~$\di$. If we denote by~$\Mdelta(\di)$ the closure of a stratum~$\M(\di)$ in~$\Mdelta_{0,S}$, then we have		
		$$\Mdelta(\di)\supset \Mdelta(\di') \;\Leftrightarrow \; \di\leq \di'\ ,$$
		where the order~$\leq$ on dissections is the one defined in Section~\ref{sec:Diss}. The closure~$\Mdelta(\di)$ is thus the union of the strata~$\M(\di')$, for~$\di'\geq \di$.
		
		For a dissection~$\di\in\mathsf{Diss}(S,\delta)$, we have the product decompositions
		$$\M(\di)\cong \prod_{p\in P(\di)} \M_{0,E(p)} \;\;\textnormal{ and }\;\; \Mdelta(\di)\cong \prod_{p\in P(\di)} \M^{\delta(p)}_{0,E(p)}$$
		which are compatible with the product decompositions (\ref{eqproductdecMMbar}).
		
		The stratum corresponding to the corolla is the open stratum~$\M_{0,S}$. For~$\di=\{c\}\in\mathsf{Diss}_1(S,\delta)$ a dissection consisting of only one chord, we get a divisor
		$$\Mdelta(\{c\})\cong \M^{\delta_0}_{0,E_0}\times\M^{\delta_1}_{0,E_1}$$
		inside~$\Mdelta_{0,S}$. These divisors are the irreducible components of~$\partial\Mdelta_{0,S}$. 
		
		\begin{ex}\leavevmode 
		\begin{enumerate}
		\item We have~$\Mdelta_{0,3}=\{*\}$.
		\item If we write~$\M_{0,4}=\mathbb{P}^1(\C)\setminus\{\infty,0,1\}$ and~$\Mbar_{0,4}=\mathbb{P}^1(\C)$, then we have~$\Mdelta_{0,4}=\mathbb{P}^1(\C)\setminus\{\infty\}$. The divisor at infinity~$\partial\Mdelta_{0,4}=\{0,1\}$ has two irreducible components, all isomorphic to a product~$\Mdelta_{0,3}\times\Mdelta_{0,3}$, indexed by the two dissection of a~$4$-gon with one chord.
		\item Figure~\ref{figureMdeltazerocinq} shows the combinatorial structure of~$\Mdelta_{0,5}$ inside~$\Mbar_{0,5}$. The curves in dashed lines are the complement~$\Mbar_{0,5}\setminus\Mdelta_{0,5}$. The five curves in straight lines are the five irreducible components of the divisor at infinity~$\partial\Mdelta_{0,5}$, indexed by the five dissection of a~$5$-gon with one chord. They bound a pentagon (shaded). The five different intersection points of these components are indexed by the five dissections of a~$5$-gon with two chords.
		\end{enumerate}
		\end{ex}

		\begin{figure}[h!!]
		\def\svgwidth{.3\textwidth}
		\begin{center}
\begin{tikzpicture}[scale=0.27]
\draw[white, fill=cyan] (-10,-10) -- (-10,0) -- (0,0) -- (-10, -10);
\draw[white, fill=white] (0:3) arc [radius=3, start angle=0, end angle= 360];
\draw[white, fill=white] (-7,-10) arc [radius=3, start angle=0, end angle= 360];

\draw[dashed] (-11.5, 10) -- (8,10);
\draw[thick] (-11.5, 0) -- (-2,0);
\draw[thick] (2, 0) -- (11.5,0);
\draw[dashed] (-8, -10) -- (11.5,-10);

\draw[thick] (-10, 11.5) -- (-10,-8);
\draw[dashed] (0, 11) -- (0,2);
\draw[dashed] (0, -2) -- (0,-11);
\draw[dashed] (10, 8) -- (10,-11);

\draw[thick] (1.5,1.5) -- (8.5,8.5);
\draw[thick] (-1.5,-1.5) -- (-8.5,-8.5);

\draw [thick] (150:3) arc [radius=3, start angle=150, end angle= 300];
\draw [dashed] (7.42,11.5) arc [radius=3, start angle=150, end angle= 300];
\draw [thick] (-7.42, -11.5) arc [radius=3, start angle=-30, end angle= 120];
\end{tikzpicture}

		\end{center}
		\caption{The combinatorial structure of~$\Mdelta_{0,5}$.}\label{figureMdeltazerocinq}
		\end{figure}
		
		\begin{rem}
		The stratification of~$\Mdelta_{0,n}$ has the same combinatorial structure as the natural stratification of an associahedron~$K_n$ of dimension~$n-3$. More precisely, there is a natural smooth embedding of~$K_n$ inside~$\Mdelta_{0,n}$ which is compatible with these stratifications (the shaded pentagon in Figure~\ref{figureMdeltazerocinq}). This is the same as Devadoss's realization of the associahedron~\cite[Definition 3.2.1]{devadossmosaic}. In that sense, Brown's moduli spaces~$\Mdelta_{0,n}$ are algebro-geometric analogs of associahedra.
		\end{rem}
		
		The dihedral decomposition morphisms
		$$\Delta_{\di}:\Grav(S, \delta)\rightarrow \Grav(\di)$$
		may be computed as (signed) residues of logarithmic forms on~$(\Mdelta_{0,S},\partial\Mdelta_{0,S})$. This is particularly interesting since~$\Mdelta_{0,S}$ is affine (Theorem~\ref{thmbrownaffine}) and we can thus use global logarithmic forms. We will give explicit formulas for these dihedral decomposition morphisms in Proposition~\ref{propresiduechords}.
	
	\subsection{The residue spectral sequence}
			In the global setting of Section~\ref{parncdstratglobal}, we prove the existence of a \textit{residue spectral sequence} which computes the cohomology of the ambient space~$\overline{X}$ in terms of the cohomology of the strata~$X(\mathfrak{s})$ and the residue morphisms. In the next paragraph, we will apply this spectral sequence to the dihedral gravity cooperad.
		
			\begin{prop}\label{propresspectralgen}
			Let~$\overline{X}$ be a smooth (not necessarily compact) complex algebraic variety and let~$\partial\overline{X}$ be a normal crossing divisor inside~$\overline{X}$, inducing a stratification
			$$\overline{X}=\bigsqcup_{\mathfrak{s}\in\mathsf{Strat}}X(\mathfrak{s})\ .$$ 
			There exists a first quadrant spectral sequence in the category of mixed Hodge structures
			\begin{equation*}
			E_1^{p,q}= \bigoplus_{\mathfrak{s}\in\mathsf{Strat}_p}H^{q-p}(X(\mathfrak{s}))(-p) \;\Longrightarrow\; H^{p+q}(\overline{X})\ ,
			\end{equation*}
			where the differential~$d_1:E_1^{p,q}\rightarrow E_1^{p+1,q}$ is the sum of the residue morphisms (\ref{eqresHodge})
			$$\mathrm{Res}_{\mathfrak{s}'}^{\mathfrak{s}}:H^{q-p}(X(\mathfrak{s}))(-p)\rightarrow H^{q-p-1}(X(\mathfrak{s}'))(-p-1)$$
			for~$\mathfrak{s}\in \mathsf{Strat}_{p}$ and~$\mathfrak{s}'\in \mathsf{Strat}_{p+1}$ such that~$\mathfrak{s}\leq \mathfrak{s}'$.
			\end{prop}
			
			\begin{proof} We first forget about mixed Hodge structures and prove the existence of the spectral sequence for the cohomology over~$\C$.
			Let us denote by~$i_{\mathfrak{s}}:\overline{X}(\mathfrak{s})\hookrightarrow \overline{X}$ the natural closed immersions. Let us write
			$$\K^{p,q}=\bigoplus_{\mathfrak{s}\in\mathsf{Strat}_p}(i_{\mathfrak{s}})_*\Omega^{q-p}_{\overline{X}(\mathfrak{s})}(\log\partial \overline{X}(\mathfrak{s}))\ .$$
			We give the collection of the~$\K^{p,q}$'s the structure of a double complex of sheaves on~$\overline{X}$.			
			The horizontal differential~$d':\K^{p,q}\rightarrow \K^{p+1,q}$ is induced by the residues
			$$(i_{\mathfrak{s}})_*\Omega^{q-p}_{\overline{X}(\mathfrak{s})}(\log\partial \overline{X}(\mathfrak{s})) \rightarrow (i_{\mathfrak{s}'})_*\Omega^{q-p-1}_{\overline{X}(\mathfrak{s}')}(\log\partial\overline{X}(\mathfrak{s}'))$$
			for~$\mathfrak{s}\in \mathsf{Strat}_{p}$ and~$\mathfrak{s}'\in \mathsf{Strat}_{p+1}$ such that~$\mathfrak{s}\leq \mathfrak{s}'$. The vertical differential~$d'':\K^{p,q}\rightarrow \K^{p,q+1}$ is induced by the exterior derivative on differential forms. One checks that we have~$d'\circ d'=0$,~$d''\circ d''=0$ and~$d'\circ d''+d''\circ d'=0$. We denote the corresponding total complex by 
			$$\K^n=\bigoplus_{p+q=n}\K^{p,q}\ .$$ 
			Using local coordinates on~$\overline{X}$, it is easy to check that we have a long exact sequence
			\begin{equation}\label{eqresolutionOmega}
			0\rightarrow \Omega^\bullet_{\overline{X}} \rightarrow \K^{0,\bullet}\rightarrow \K^{1,\bullet}\rightarrow \K^{2,\bullet}\rightarrow \cdots\ , 
			\end{equation}
			which induces a quasi-isomorphism~$\Omega^\bullet_{\overline{X}}\isom \K^\bullet$.
			The holomorphic Poincar\'{e} lemma implies that we have a quasi-isomorphism~$\C_{\overline{X}}\isom \Omega^\bullet_{\overline{X}}$, hence we get an isomorphism
			$$H^q(\overline{X},\C)\cong \mathbb{H}^q(\overline{X},\K^\bullet)\ .$$
			Now, the hypercohomology spectral sequence for the double complex~$\K^{\bullet,\bullet}$ filtered by the columns is exactly
			$$E_1^{p,q}=\bigoplus_{\mathfrak{s}\in \mathsf{Strat}_p}\mathbb{H}^q(\overline{X}(\mathfrak{s}),\Omega^{\bullet-p}_{\overline{X}(\mathfrak{s})}(\log\partial \overline{X}(\mathfrak{s}))) \;\Longrightarrow \; H^{p+q}(\overline{X},\C).$$ 
			Taking into account the isomorphisms~$\mathbb{H}^q(\overline{X}(\mathfrak{s}),\Omega^{\bullet-p}_{\overline{X}(\mathfrak{s})}(\log\partial \overline{X}(\mathfrak{s}))) \isom H^{q-p}(X(\mathfrak{s}),\C)$, one gets the desired spectral sequence.
			
 In order to prove that this spectral sequence is defined over~$\Q$, it is convenient to work in the category of perverse sheaves. We let~$u_{\mathfrak{s}}:X(\mathfrak{s})\hookrightarrow \overline{X}$ denote the natural locally closed immersions. We replace (\ref{eqresolutionOmega}) by the following long exact sequence in the category of perverse sheaves on~$\overline{X}$, where~$d$ denotes the complex dimension of~$\overline{X}$:
			$$0\rightarrow \Q_{\overline{X}}[d] \rightarrow u_*\Q_X[d] \rightarrow \bigoplus_{\mathfrak{s}\in\mathsf{Strat}_1}(u_{\mathfrak{s}})_*\Q_{X(\mathfrak{s})}[d-1] \rightarrow \bigoplus_{\mathfrak{s}\in\mathsf{Strat}_2}(u_{\mathfrak{s}})_*\Q_{X(\mathfrak{s})}[d-2] \rightarrow \cdots\ .$$
			Taking the hypercohomology spectral sequence and shifting all the degrees by~$d$ gives the result.
			\item The proof via perverse sheaves can be copied in the category of mixed Hodge modules~\cite{saitomodulesdehodgepolarisables}, see~\cite[Section~$14$]{peterssteenbrink}, which proves the compatibility with mixed Hodge structures.
			\end{proof}

	\subsection{Purity and freeness}
	
			We start with a classical theorem on the cohomology of the moduli spaces~$\M_{0,S}$.
		
			\begin{thm}\label{thmpurityM}
			For every integer~$k$ and every set~$S$, the cohomology group~$H^k(\M_{0,S})$ is pure Tate of weight~$2k$.
			\end{thm}	
			
			\begin{proof}
			Since the moduli space~$\M_{0,S}$ is a complement of a union of hyperplanes in the affine space~$\C^{n-3}$ by (\ref{eqMcomplementarrangement}), this is a consequence of a general result on complements of hyperplane arrangements~\cite{lehrerladic,shapiropure,kimweights}. See also Getzler's proof~\cite[Lemma 3.12]{getzlermodulispacesgenuszero} which only uses Arnol'd's result~\cite{arnold}.
			\end{proof}
			
			The residue spectral sequence of the previous paragraph now allows us to compute the cohomology of Brown's moduli spaces~$\Mdelta_{0,S}$ in term of the cohomology of  the spaces~$\M_{0,S}$.

			\begin{prop}
			There exists a first quadrant spectral sequence in the category of mixed Hodge structures
			\begin{equation}\label{eqssres}
			E_1^{p,q}= \bigoplus_{\di\in \mathsf{Diss}_p(S,\delta)}H^{q-p}(\M(\di))(-p) \;\Longrightarrow\; H^{p+q}(\Mdelta_{0,S})\ ,
			\end{equation}
			where the differential~$d_1:E_1^{p,q}\rightarrow E_1^{p+1,q}$ is the sum of the residue morphisms
			$$\mathrm{Res}^\di_{\di'}:H^{q-p}(\M(\di))(-p) \rightarrow H^{q-p-1}(\M(\di'))(-p-1)\ ,$$
			for~$\di\in \mathsf{Diss}_{p}(S,\delta)$ and~$\di'\in \mathsf{Diss}_{p+1}(S,\delta)$ such that~$\di\leq \di'$. 
			\end{prop}
			
 			\begin{proof}
			This is a direct application of  Proposition~\ref{propresspectralgen} to the case~$\overline{X}=\Mdelta_{0,S}$ with the stratification (\ref{eqstratdelta}).
			\end{proof}
			
			The~$q$-th row of the first page~$E_1$ of the spectral sequence (\ref{eqssres}) looks like
			\begin{equation}\label{eqcobarcohomology}
			0\rightarrow H^q(\M_{0,S}) \rightarrow \bigoplus_{\di\in\mathsf{Diss}_1(S,\delta)}H^{q-1}(\M(\di))(-1) \rightarrow \bigoplus_{\di\in \mathsf{Diss}_2(S,\delta)} H^{q-2}(\M(\di))(-2) \rightarrow \cdots \ .
			\end{equation}
			
			\begin{prop}\label{propequivsscobar}
			The direct sum of the rows~$E_1^{\bullet,q}$ of the first page of the spectral sequence (\ref{eqssres}) is, up to a Tate twist~$(-1)$, the dihedral cobar construction of the (desuspension of the) dihedral gravity cooperad.
			\end{prop}
			
			\begin{proof}
			After twisting by~$(-1)$, the direct sum of the complexes (\ref{eqcobarcohomology}) can be written as
			$$0\rightarrow s^{-1}\mathcal{C}(S, \delta) \rightarrow \bigoplus_{\di\in\mathsf{Diss}_1(S,\delta)} s^{-1}\mathcal{C}(\di) \rightarrow \bigoplus_{\di\in\mathsf{Diss}_2(S,\delta)} s^{-1}\mathcal{C}(\di) \rightarrow \cdots\ ,$$
			where the arrows are (signed) infinitesimal decomposition morphisms. We leave it to the reader to check that the sign conventions are consistent. 
			\end{proof}
			
			We now turn to the degeneration of this spectral sequence.
			
			\begin{prop}\label{lemdegen}
			The spectral sequence (\ref{eqssres}) degenerates at the second page~$E_2$, that is~$E_{\infty}=E_2$.
			\end{prop}
			
			\begin{proof}
			As a consequence of Theorem~\ref{thmpurityM} and the K\"{u}nneth formula,~$H^{q-p}(\M(\di))$ is pure Tate of weight~$2(q-p)$ for every dissection~$\di\in \mathsf{Diss}_p(S,\delta)$, and hence~$H^{q-p}(\M(\di))(-p)$ is pure Tate of weight~$2(q-p)+2p=2q$. The differential~$d_r:E_r^{p,q}\rightarrow E_r^{p+r,q-r+1}$ thus maps a pure Hodge structure of weight~$2q$ to a pure Hodge structure of weight~$2(q-r+1)$, and is zero for~$r\geq 2$ by Lemma~\ref{lemappendixzero}.
			\end{proof}
			
			In the next proposition, we prove the equivalence between two statements: a geometric statement~$(i)$, namely the purity of the Hodge structure on the cohomology of Brown's moduli spaces, and an algebraic statement~$(ii)$, namely the freeness of the dihedral gravity cooperad.	In the next section, we will prove the algebraic statement~$(ii)$ and derive the geometric statement~$(i)$. We nevertheless state this proposition as an equivalence to convince the reader that the mathematical content of the two statements is \emph{essentially the same}.
			
			\begin{thm}\label{propequivalence}
			The following statements are equivalent:
			\begin{enumerate}[(i)]
			\item for every integer~$k$ and every dihedral set~$(S,\delta)$, the cohomology group~$H^k(\Mdelta_{0,S})$ is pure Tate of weight~$2k$;
			\item the dihedral gravity cooperad is cofree. 
			\end{enumerate}
When they are true,  there is a (non-canonical) isomorphism between the dihedral gravity cooperad and the cofree dihedral cooperad on the dihedral module
			$$(S,\delta)\mapsto \det(S)\otimes H^{\bullet+n-3}(\Mdelta_{0,S})(-1)\ .$$
			\end{thm}
			
			\begin{proof} 
			 Let us denote by~$A$ the filtration on the cohomology of~$\Mdelta_{0,S}$ that is induced by the spectral sequence (\ref{eqssres}). It is a filtration by mixed Hodge sub-structures. By Proposition~\ref{lemdegen}, we get	
			 at the second page:
			$$E_2^{p,q}=\mathrm{gr}_A^pH^{p+q}(\Mdelta_{0,S})\ .~$$
			By 
			the proof of Proposition~\ref{lemdegen}, the space~$E_2^{p,q}$ is pure Tate of weight~$2q$. Thus,~$(i)$ is equivalent to the fact that for every~$(S,\delta)$, the spectral sequence (\ref{eqssres}) satisfies~$E_2^{p,q}=0$ for~$p>0$. This is the same as requesting that each row~$E_1^{\bullet,q}$ is exact except possibly at~$\bullet=0$. According to Proposition~\ref{propequivsscobar} and Proposition~\ref{propfreecobar}, this is equivalent to~$(ii)$, and we have proved the equivalence between statements~$(i)$ and~$(ii)$. Assuming them, we see that~$H^k(\Mdelta_{0,S})=E_2^{0,k}$ is the kernel of the map
			$$H^k(\M_{0,S}) \xrightarrow{\bigoplus \Delta_\di}  \bigoplus_{\di\in \mathsf{Diss}_1(S,\delta)} H^{k-1}(\M(\di))(-1) \, ,$$
			hence the result about the cogenerators of the dihedral gravity cooperad, after a degree shift and an operadic suspension.
			\end{proof}
			
			\begin{rem}
			We can also apply the residue spectral sequence to the case~$\overline{X}=\Mbar_{0,S}$, with the stratification (\ref{eqstratMbar}). We then get a spectral sequence in the category of mixed Hodge structures
			$$E_1^{p,q}= \bigoplus_{\tr\in \mathsf{Tree}_p(S)}H^{q-p}(\M(\tr))(-p) \;\Longrightarrow\; H^{p+q}(\Mbar_{0,S})\ ,$$
			which degenerates at the second page~$E_2$. It is a classical fact that the odd cohomology groups of~$\Mbar_{0,S}$ are zero, and that for every~$k$,~$H^{2k}(\Mbar_{0,S})$ is pure Tate of weight~$2k$. Thus, the degeneration of the spectral sequence gives rise to a long exact sequence
			$$0\rightarrow H^k(\M_{0,S})\rightarrow \bigoplus_{\tr\in \mathsf{Tree}_1(S)}H^{k-1}(\M(\tr))(-1)\rightarrow\cdots\rightarrow \bigoplus_{\tr\in \mathsf{Tree}_k(S)}H^0(\M(\tr))(-k)\rightarrow H^{2k}(\Mbar_{0,S})\rightarrow 0\ .$$
			After dualizing and performing an operadic suspension, this long exact sequence gives a quasi-isomorphism from the cyclic hypercommutative operad 
			$S\mapsto H_\bullet(\Mbar_{0,S})$
			to the cyclic bar construction of the cyclic gravity operad. Under the bar-cobar adjunction, this corresponds to Getzler's quasi-isomorphism~\cite[Theorem~$4.6$]{getzlermodulispacesgenuszero}, which proves the Koszul duality between the cyclic hypercommutative operad and the cyclic gravity operad.
			\end{rem}

\section{The dihedral gravity cooperad is cofree}

	We prove that the dihedral gravity cooperad is cofree by using explicit formulas describing the cohomology of the moduli spaces~$\M_{0,S}$. The main point consist in showing that the filtration given by residual chords is the coradical filtration of the dihedral gravity cooperad. We then derive geometric consequences for Brown's moduli spaces~$\Mdelta_{0,S}$ and a new proof of a theorem of Salvatore--Tauraso.

	\subsection{Conventions}

		In this section, we will work with explicit formulas for the decomposition morphisms in the dihedral gravity cooperad. For reasons of signs, it is easier to work with its desuspension~$\mathcal{C}$, whose underlying~$\mathsf{Dih}$-module is given by
		$$\mathcal{C}(S,\delta)=H^{\bullet-1}(\M_{0,S})(-1)\ .$$
		We use the notation~$\mathcal{C}(S,\delta)$ instead of~$\mathcal{C}(S)$ because we will use a spanning set and a filtration for this space that depend on the choice of a dihedral structure.
		
		For~$\di\in\mathsf{Diss}_k(S,\delta)$ a dissection of cardinality~$r$, we will always choose an ordering~$P(\di)=\{p_0,\ldots,p_k\}$ and write~$E_i:=E(p_i)$ for the set of edges of the sub-polygons~$p_i$,~$\delta_i:=\delta(p_i)$ for the induced dihedral orders. The ordering of~$P(\di)$ gives a trivialization~$\det(P(\di))\simeq\Q$ and hence we can simply write
		$$\Delta_\di:\mathcal{C}(S,\delta)\rightarrow\mathcal{C}(\di)=\mathcal{C}(E_0,\delta_0)\otimes\cdots\otimes\mathcal{C}(E_k,\delta_k)$$
		for the dihedral decompositions (\ref{eqdeltadet}).	
	
	\subsection{Cohomology of the moduli spaces~$\M_{0,S}$}\label{parcohomologyMdelta}
			
		Let~$S$ be a finite set of cardinality~$n\geq 3$ and let~$\delta$ be a dihedral structure on~$S$. We first recall Brown's presentation of the cohomology algebra of the moduli space~$\M_{0,S}$, which is well suited for computing residues on~$\Mdelta_{0,S}$. For any chord~$c$ of~$(S,\delta)$, there exists a global holomorphic function~$u_c\in \mathcal{O}(\Mdelta_{0,S})$ such that the divisor~$\Mdelta(\{c\})$ is defined by the vanishing of~$u_c$:
		$$\Mdelta(\{c\})=\{u_c=0\}.$$

		We then define the following closed logarithmic differential~$1$-form		on~$\M_{0,S}$: 
		$$\omega_c:=\dfrac{1}{2\pi i}\dfrac{du_c}{u_c}\ .$$
 We denote by the same symbol~$\omega_c$ its class in~$H^1(\M_{0,S})$.	
		
		\begin{prop}\cite[Proposition 6.2]{brownPhD}
		The cohomology algebra~$H^\bullet(\M_{0,S})$ is generated by the classes~$\omega_c$. In other words,~$\mathcal{C}(S,\delta)$ is spanned by monomials~$\omega_{c_1}\wedge\cdots\wedge\omega_{c_k}$ for some chords~$c_1,\ldots,c_k$ of~$(S,\delta)$.
		\end{prop}		
		
		We note that every differential form~$\omega_{c_1}\wedge\cdots\wedge\omega_{c_k}$ is a logarithmic form on~$(\Mdelta_{0,S},\partial\Mdelta_{0,S})$.
		
		\begin{rem}
		It is convenient to represent a monomial~$\omega_{c_1}\wedge\cdots\wedge\omega_{c_k}$, up to a sign, by the picture of the set of chords~$\{c_1,\ldots,c_k\}$, as in Figure~\ref{figuremonomial}, where the chords are pictured in dashed lines.
		\end{rem}
		
		\begin{figure}[h!!]
		\def\svgwidth{.26\textwidth}
		\begin{center}
\begin{tikzpicture}[scale=0.27]

\draw[thick] (0:10) -- (30:10) -- (60:10) -- (90:10) -- (120:10) -- (150:10) -- (180:10) 
-- (210:10) -- (240:10) -- (270:10) -- (300:10) -- (330:10) -- (360:10);

\draw[thick, dashed] (60:10) -- (270:10) ;
\draw[thick, dashed] (60:10) -- (0:10) ;
\draw[thick, dashed] (30:10) -- (300:10) ;

\draw[thick, dashed] (90:10) -- (270:10) ;

\draw[thick, dashed] (150:10) -- (270:10);
\draw[thick, dashed] (120:10) -- (210:10);

\end{tikzpicture}

		\end{center}
		\caption{A monomial (up to a sign) in~$H^6(\M_{0,10})$.}\label{figuremonomial}
		\end{figure}
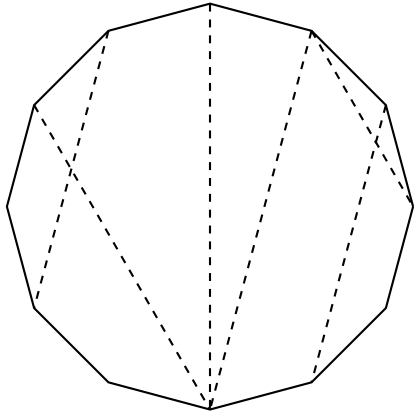
		
		\begin{rem}
		The ideal of relations between the classes~$\omega_c$ in~$H^\bullet(\M_{0,S})$ can be described in pure combinatorial terms with sets of chords that cross completely, see~\cite[Proposition 6.2]{brownPhD}. Surprisingly enough, this will not play any role in the sequel.
		\end{rem}
		
		The decomposition morphisms of the dihedral gravity cooperad are easily computed in terms of the symbols~$\omega_c$. They are completely determined by the infinitesimal ones, which correspond to dissections made up of one chord.
		
		\begin{prop}\label{propresiduechords}
		Let~$c$ be a chord which dissects~$(S,\delta)$ into two polygons~$p_0$ and~$p_1$. The corresponding dihedral decomposition morphism 
		$$\Delta_{\{c\}}:\mathcal{C}(S,\delta)\rightarrow  \mathcal{C}(E_0,\delta_0)\otimes \mathcal{C}(E_1,\delta_1)$$ 
		is given by 
		\begin{enumerate}[1)]
		\item~$\Delta_{\{c\}}(\omega_{c_1}\wedge\cdots\wedge\omega_{c_k})=0\;$ if~$c\notin\{c_1,\ldots,c_k\}$;
		\item~$\Delta_{\{c\}}(\omega_c\wedge \omega_{c_1}\wedge\cdots \wedge \omega_{c_{k}})=0\;$ if~$c$ crosses some chord~$c_i$, for~$i=1,\ldots,k$;
		\item~$\Delta_{\{c\}}( X_0\wedge\omega_c\wedge X_1)=X_0\otimes X_1\;$ if~$X_i$ is a monomial formed with chords in~$p_i$,~$i=0,1$.
		\end{enumerate}
		\end{prop}
		
		\begin{proof}\leavevmode
		
		\begin{enumerate}[1)]
		\item This is because the differential form~$\omega_{c_1}\wedge\cdots\wedge\omega_{c_k}$ has no pole along~$\Mdelta(\{c\})$ if~$c\notin\{c_1,\ldots,c_k\}$.
		\item By definition of the residue morphisms,~$\Delta_{\{c\}}(\omega_c\wedge \omega_{c_1}\wedge\cdots \wedge \omega_{c_{k}})$ is, up to a sign, the restriction of the differential form~$\omega_{c_1}\wedge\cdots \wedge \omega_{c_{k}}$ on~$\Mdelta(\{c\})$. If~$c$ crosses some chord~$c_i$, for~$i=1,\ldots,k$, then the proof of~\cite[Lemma 2.6]{brownPhD} implies that~$\omega_{c_i}$ is zero when restricted to~$\Mdelta(\{c\})$, hence the result.
		\item Let us denote by~$a-1$ and~$b-1$ the respective degrees of~$X_0$ and~$X_1$, so that they respectively live in degree~$a$ and~$b$ in~$\mathcal{C}$. Then we get~$X_0\wedge\omega_c\wedge X_1=(-1)^{a-1}\omega_c\wedge X_0\wedge X_1$, whose residue on~$\Mdelta(\{c\})$ is the restriction of~$(-1)^{a-1}X_0\wedge X_1$ on~$\Mdelta(\{c\})$. Note that the sign~$(-1)^{a-1}$ is canceled by the Koszul sign in the definition (\ref{eqRes12}) of~$\Delta_{\{c\}}$. By the proof of~\cite[Lemma 2.6]{brownPhD}, the pullback morphism~$\mathcal{O}(\Mdelta_{0,S})\rightarrow\mathcal{O}(\M^{\delta_0}_{0,E_0})\otimes\mathcal{O}(\M^{\delta_1}_{0,E_1})$ is given by~$u_{c_0}\mapsto u_{c_0}\otimes 1$ and~$u_{c_1}\mapsto 1\otimes u_{c_1}$, for~$c_i$ a chord in~$p_i$ for~$i=0,1$. The result follows.
		\end{enumerate}		
		\end{proof}
		
		\begin{rem}
		The formula of Proposition~\ref{propresiduechords} 3), is easy to represent pictorially: if~$c$ is a chord that is not crossed by any other, applying~$\Delta_{\{c\}}$ has the effect of cutting the polygon along~$c$ into two parts, see Figure~\ref{figuredecomposition}.
		\end{rem}
		
		\begin{figure}[h!!]
$$	
\vcenter{\hbox{\begin{tikzpicture}[scale=0.27]
\draw[thick] (0:10) -- (30:10) -- (60:10) -- (90:10) -- (120:10) -- (150:10) -- (180:10) 
-- (210:10) -- (240:10) -- (270:10) -- (300:10) -- (330:10) -- (360:10);

\draw[thick, dashed] (60:10) -- (270:10) ;
\draw[thick, dashed] (60:10) -- (0:10) ;
\draw[thick, dashed] (30:10) -- (300:10) ;

\draw[thick, dashed, blue] (90:10) -- (270:10) node[midway,  left] {\scalebox{1}{$\mathbf{c}$\ }};

\draw[thick, dashed] (150:10) -- (270:10);
\draw[thick, dashed] (120:10) -- (210:10);
\end{tikzpicture}}}
\quad \xymatrix{\ar@{|->}[r]^{\Delta_{\{c\}}}&}\quad
	%\hspace{1.7cm}
%		\def\svgwidth{.19833\textwidth}	
\vcenter{\hbox{\begin{tikzpicture}[scale=0.27]

\draw[thick]  (90:10) -- (120:10) -- (150:10) -- (180:10) 
-- (210:10) -- (240:10) -- (270:10) -- (90:10) ;

\draw[thick, dashed] (150:10) -- (270:10);
\draw[thick, dashed] (120:10) -- (210:10);
\end{tikzpicture}}}
\quad \otimes \quad
%		\hspace{1.2cm}
%		\def\svgwidth{.130083\textwidth}	
\vcenter{\hbox{\begin{tikzpicture}[scale=0.27]

\draw[thick] (0:10) -- (30:10) -- (60:10) -- (90:10) -- (270:10) -- (300:10) -- (330:10) -- (360:10);

\draw[thick, dashed] (60:10) -- (270:10) ;
\draw[thick, dashed] (60:10) -- (0:10) ;
\draw[thick, dashed] (30:10) -- (300:10) ;
\end{tikzpicture}}}
$$
		\caption{The dihedral decomposition~$\Delta_{\{c\}}:\mathcal{C}(10,\delta)\rightarrow\mathcal{C}(6,\delta)\otimes\mathcal{C}(6,\delta)$ applied to a monomial.}\label{figuredecomposition}
		\end{figure}
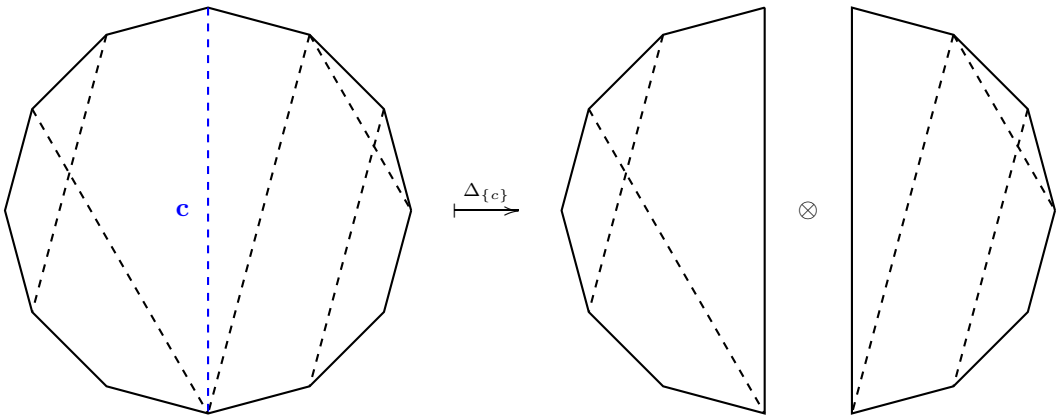

	\subsection{The residual filtration}
	
		\begin{defi}[Residual chord] Let~$\{c_1,\ldots,c_k\}$ be a set of chords of a polygon~$(S,\delta)$. We say that~$c_i$ is a \textit{residual chord} in~$\{c_1,\ldots,c_k\}$ if~$c_i$ is not crossed by any~$c_j$, for~$j\neq i$. 
		\end{defi}

 		\begin{defi}[Residual filtration]
		For every integer~$r$, we denote by 
		$$R_r \mathcal{C}(S,\delta)\subset \mathcal{C}(S,\delta)$$
		the subspace spanned by monomials~$\omega_{c_1}\wedge\cdots\wedge\omega_{c_k}$ with at most~$r$ residual chords in~$\{c_1,\ldots,c_k\}$. This gives a finite filtration
		$$0=R_{-1}\mathcal{C}(S,\delta)\subset R_0\mathcal{C}(S,\delta) \subset R_1\mathcal{C}(S,\delta)\subset \cdots  \subset\mathcal{C}(S,\delta)\ ,$$ called the \emph{residual filtration}. 
		\end{defi}

		\begin{lem}\label{lemresiduefiltration}
		For a dissection~$\mathfrak{d}\in \mathrm{Diss}_k(S,\delta)$  of~$(S,\delta)$ of cardinality~$k$, the dihedral decomposition 
		$$\Delta_\di: \mathcal{C}(S,\delta) \rightarrow \mathcal{C}(\di)=\mathcal{C}(E_0,\delta_0)\otimes\cdots\otimes\mathcal{C}(E_k,\delta_k)$$
		sends~$R_r\mathcal{C}(S,\delta)$ to~$R_{r-k} \mathcal{C}(\di)$.
		\end{lem}
		
		\begin{proof}
Since any decomposition map can be obtained by iterating infinitesimal decomposition maps, it is enough to do the case~$k=1$, which follows from Proposition~\ref{propresiduechords}: applying~$\Delta_{\{c\}}$ to a monomial either gives zero or erases a residual chord from the monomial.
		\end{proof}
		
		\begin{ex}
		In Figure~\ref{figuredecomposition}, the left-hand side lives in~$R_2\mathcal{C}(10,\delta)$ and the right-hand side lives in~$R_0\mathcal{C}(6,\delta)\otimes R_1\mathcal{C}(6,\delta)$.
		\end{ex}
		
		\begin{thm}\label{thmmain}
		For every integer~$r$ and every dihedral set~$(S,\delta)$, the morphism
		$$\Phi:\mathrm{gr}_r^R\, \mathcal{C}(S,\delta)\xrightarrow{\bigoplus \Delta_{\di}} \bigoplus_{\di\in\mathrm{Diss}_r(S,\delta)} R_0\mathcal{C}(\di)~$$
		is an isomorphism.
		\end{thm}
		
		We postpone the proof of this theorem to Section~\ref{parmainresult}, after we have introduced a technical tool.
		
	\subsection{The forgetful maps}
	
		Let~$S$ be a finite set and~$S'\subset S$ be a subset. This inclusion gives rises to a forgetful morphism
		$$f:\M_{0,S}\rightarrow \M_{0,S'}$$
		and hence a pullback in cohomology
		\begin{equation}\label{eqpullback}
		f^*:H^\bullet(\M_{0,S'})\rightarrow H^\bullet(\M_{0,S})\ , 
		\end{equation}
		which is a map of graded algebras. Now suppose that we are given a dihedral structure~$\delta$ on~$S$ and let~$\delta'$ be the induced dihedral structure on~$S'$. We view~$(S',\delta')$ as the decorated polygon obtained by contracting the sides of~$(S,\delta)$ that are not in~$S'$.  For a chord~$c$  of~$(S,\delta)$ and  a chord~$c'$ of~$(S',\delta')$, we write~$c\leadsto c'$ if this contraction transforms~$c$ into~$c'$.
				
		\begin{lem}\label{lempullbackformula}\leavevmode
		\begin{enumerate}
		\item The pullback morphism~$f^*$ is given, for~$c'$ a chord of~$(S',\delta')$, by
		$$f^*(\omega_{c'}) = \sum_{c\leadsto c'} \omega_{c}\ .$$
		\item The pullback morphism~$f^*$ is compatible with the residual filtration~$R$.
		\end{enumerate}
		\end{lem}
		
		\begin{proof}\leavevmode
		\begin{enumerate}
		\item At the level of global functions, the pullback~$\mathcal{O}(\M_{0,S'})\rightarrow \mathcal{O}(\M_{0,S})$ is computed in~\cite[Lemma 2.9]{brownPhD}, and is given by 
		$$u_{c'}\mapsto \prod_{c\leadsto c'} u_c\ .$$
		The result then follows from taking the logarithmic derivative.
		\item According to~$(1)$, the pullback of a monomial is given by
		$$f^*(\omega_{c'_1}\wedge\cdots\wedge\omega_{c'_k}) = \sum_{\{c_i\leadsto c'_i\}}\omega_{c_1}\wedge\cdots\wedge \omega_{c_k}.$$
		By construction, every set~$\{c_1,\ldots,c_k\}$ contains at most as many residual chords as~$\{c'_1,\ldots,c'_k\}$, hence the result.
		\end{enumerate}
		\end{proof}
		
	\subsection{A technical lemma}
	
		Let us fix a polygon~$(S,\delta)$.	 Let~$(E, \delta_E)$ be an inscribed polygon inside~$(S,\delta)$, that is a polygon whose sides are either sides of~$(S,\delta)$ or chords of~$(S,\delta)$, see Figure~\ref{figureexampletech}. We let~$E_{\textnormal{sides}}\subset E$ and~$E_{\textnormal{chords}}\subset E$ denote the set of sides of~$(E, \delta_E)$ which are respectively sides of~$(S,\delta)$ and chords of~$(S,\delta)$. In such a situation, we have a partition
		$$S\setminus E_{\textnormal{sides}} = \bigsqcup_{c\in E_{\textnormal{chords}}} S_c~$$
		into components~$S_c$ delimited by~$c$, that are outside of the inscribed polygon~$(E,\delta_E)$, and connected with respect to the dihedral order~$\delta$.\\
		
		For every chord~$c\in E_{\textnormal{chords}}$, let us choose a \emph{matching side}~$s_c\in S_c$, and write 
		$$S':=E_{\textnormal{sides}}\sqcup \{s_c \, ,\, c\in E_{\textnormal{chords}}\} \,\subset S\ .$$
		We let~$\delta'$ be the dihedral structure on~$S'$ induced by~$\delta$. Identifying a chord~$c$ and the matching side~$s_c$ gives rise to natural dihedral isomorphism~$(E,\delta_E)\cong (S',\delta')$. 
		
		\begin{ex}
		In Figure~\ref{figureexampletech}, the inscribed polygon is shaded and pictured in blue, with~$E_{\textnormal{chords}}=\{c_1,c_2,c_3\}$ and a possible choice of matching sides~$s_{c_1}$,~$s_{c_2}$,~$s_{c_3}$.
		\end{ex}
		
		\begin{figure}[h!!]
		\def\svgwidth{.35\textwidth}
		\begin{center}
\begin{tikzpicture}[scale=0.27]

\draw[thin] (0:10) -- (30:10) -- (60:10) -- (90:10) -- (120:10) -- (150:10) -- (180:10) 
-- (210:10) -- (240:10) -- (270:10) -- (300:10) -- (330:10) -- (360:10);

\draw[very thick, blue, fill=lightgray] (0:10) -- (90:10) -- (120:10) -- (240:10) -- (270:10) -- (300:10) -- (0:10);

\node[blue] at (180:4) {\scalebox{1}{$\mathbf{c_1}$}};
\node[blue] at (45:6) {\scalebox{1}{$\mathbf{c_2}$}};
\node[blue] at (330:7.3) {\scalebox{1}{$\mathbf{c_3}$}};

\draw[ultra thick] (0:10) -- (30:10) node[midway, right] {\scalebox{1}{$\mathbf{s_{c_2}}$}};
\draw[ultra thick] (150:10) -- (180:10) node[midway, above left] {\scalebox{1}{$\mathbf{s_{c_1}}$}};
\draw[ultra thick] (300:10) -- (330:10) node[midway, below right] {\scalebox{1}{$\mathbf{s_{c_3}}$}};

\end{tikzpicture}

		\end{center}
		\caption{An inscribed polygon and a possible choice of matching sides.}\label{figureexampletech}
		\end{figure}
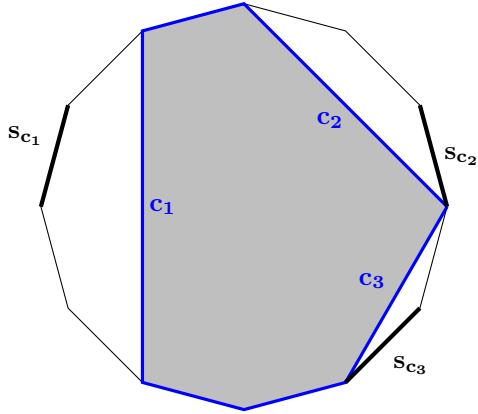
		
		The construction of the previous paragraph gives rise to a pullback morphism (\ref{eqpullback}) that we denote by 
		\begin{equation}\label{eqpsi}
		\psi:\mathcal{C}(E,\delta_E)\cong \mathcal{C}(S',\delta') \rightarrow \mathcal{C}(S,\delta)\ .
		\end{equation}

		\begin{lem}\label{lemtechnical}
		Let~$X\in \mathcal{C}(E,\delta_E)$ be a monomial formed with chords of~$(E,\delta_E)$, and let us denote by the same letter~$X$ the corresponding monomial viewed in~$\mathcal{C}(S,\delta)$. Then~$\psi(X)-X$ can be written as a sum of monomials~$\omega_{c_1}\wedge\cdots\wedge\omega_{c_k}$ for which some chord~$c_i$ crosses a chord in~$E_{\mathrm{chords}}$.
		\end{lem}
		
		\begin{proof}
		It is enough to do the proof for a monomial~$X=\omega_c$. We do the proof in the case where~$E_{\mathrm{chords}}$ only contains one element~$c_1$ corresponding to a side~$s_{c_1}\in S$, the general case being similar. The formula for~$\psi(\omega_c)$ is given in Lemma~\ref{lempullbackformula}. If~$c$ and~$c_1$ do not have a vertex in common, then~$\psi(\omega_c)=\omega_c$. Else, let us denote by~$v_1$ the common vertex of~$c$ and~$w$ the other vertex. We use the notation~$c=v_1w$. We then have 
		$$\psi(\omega_c)=\sum_{v}\omega_{vw}\ ,$$ 
		where the sum ranges over the vertices~$v\in S_{c_1}$ that are between~$v_1$ and the first vertex of~$s_{c_1}$. For such vertices~$v$, the chord~$vw$ crosses~$c_1$, except if~$v=v_1$. The claim follows.
		\end{proof}
			
		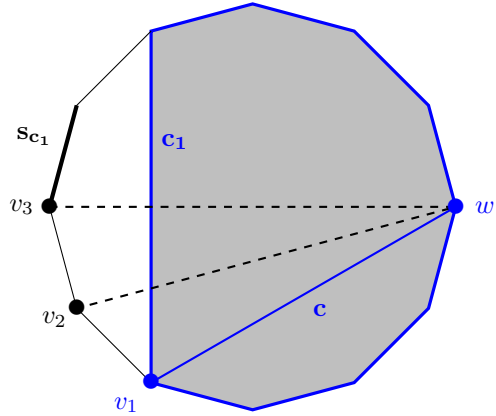
\begin{figure}[h!!]
\begin{tikzpicture}[scale=0.27]
\draw[very thick, blue, fill=lightgray] (0:10) -- (30:10) -- (60:10) -- (90:10) -- (120:10) -- (240:10) -- (270:10) -- (300:10) -- (330:10) -- (0:10);

\draw[thin] (120:10) -- (150:10) -- (180:10) -- (210:10) -- (240:10); 

\node[blue] at (140:5) {\scalebox{1}{$\mathbf{c_1}$}};

\draw[ultra thick] (150:10) -- (180:10) node[midway, above left] {\scalebox{1}{$\mathbf{s_{c_1}}$}};

\node[blue] at (240:10) {\scalebox{1.5}{$\bullet$}};
\node[blue, below left] at (240:10.3) {\scalebox{1}{$v_1$}};

\draw[thick, blue] (0:10) -- (240:10) node[midway, below right] {\scalebox{1}{$\mathbf{c}$}}; 

\node at (210:10) {\scalebox{1.5}{$\bullet$}};
\node[ left] at (212.5:10.3) {\scalebox{1}{$v_2$}};
\node at (180:10) {\scalebox{1.5}{$\bullet$}};
\node[ left] at (180:10.3) {\scalebox{1}{$v_3$}};

\draw[thick, dashed] (180:10) -- (0:10)  -- (210:10); 

\node[blue] at (0:10) {\scalebox{1.5}{$\bullet$}};
\node[blue, right] at (0:10.5) {\scalebox{1}{$w$}};

\end{tikzpicture}

		\begin{center}
		\end{center}
		\caption{Illustration of the proof of Lemma~\ref{lemtechnical}.}\label{figurelemtech}
		\end{figure}
		
		\begin{ex}
		Figure~\ref{figurelemtech} illustrates the proof of Lemma~\ref{lemtechnical}: the inscribed polygon~$(E,\delta_E)$ is shaded pictured in blue. We have~$\psi(\omega_{v_1w})=\omega_{v_1w}+\omega_{v_2w}+\omega_{v_3w}$.
		\end{ex}

	\subsection{Proof of the main result}\label{parmainresult}
	
		We now have all the tools to prove Theorem~\ref{thmmain}.
		
		\begin{proof}[Proof of Theorem~\ref{thmmain}]
		To prove this theorem, we will construct the inverse morphism~$\Psi$. To this aim, let us make some ordering conventions to make the signs explicit. For a dissection~$\di\in\mathsf{Diss}_r(S,\delta)$, we will choose compatible orderings 
		\begin{equation}\label{eqtechordering}
		\di=\{c_1,\ldots,c_r\} \;\textnormal{ and } \; P(\di)=\{p_0,\ldots,p_r\}
		\end{equation}
		 that obey the following constraint. Let~$\widetilde{\tr}$ be the tree obtained by removing the leaves (external vertices) of the tree~$\tr$ corresponding to~$\di$. The chords~$c_i$ label the edges of~$\widetilde{\tr}$ and the polygons~$p_i$ label the vertices of~$\widetilde{\tr}$. We choose the orderings (\ref{eqtechordering}) such that for every~$j=1,\ldots,r-1$, deleting the edges labeled by~$c_i$ for~$i=1,\ldots,j$, only disconnects the vertices~$p_i$ for~$i=0,\ldots,j-1$.
		
		An element of~$\mathrm{gr}^R_r\mathcal{C}(S,\delta)$ can be represented as a sum of elements~$X_0\wedge \omega_{c_1} \wedge X_1\wedge \omega_{c_2} \wedge \cdots \wedge \omega_{c_r} \wedge X_r$, for some dissection~$\di=\{c_1,\ldots,c_r\}\in\mathsf{Diss}_r(S,\delta)$, with~$X_i\in R_0\mathcal{C}(E_i,\delta_i)$. According to the constraint we put on the orderings (\ref{eqtechordering}), the image of such an element by~$\Phi$ is
		\begin{equation}\label{eqproofPhi}
		\Delta_\di(X_0\wedge \omega_{c_1} \wedge X_1\wedge \omega_{c_2} \wedge \cdots \wedge \omega_{c_r} \wedge X_r)= X_0\otimes\cdots\otimes X_r
		\end{equation}
		by repeated applications of Proposition~\ref{propresiduechords}.

		For every~$i=0,\ldots,r$, we let 
		$$\psi_i : \mathcal{C}(E_i,\delta_i) \rightarrow \mathcal{C}(S,\delta)$$
		denote the pullback map (\ref{eqpsi}) defined in the previous paragraph, corresponding to the inscribed polygon~$p_i=(E_i,\delta_i)$ and any choice of matching sides~$s_c$ for~$c\in (E_i)_{\textnormal{chords}}$.
		 
		Let us recall that we have 
		$$R_0\mathcal{C}(\di)= R_0\mathcal{C}(E_0,\delta_0)\otimes\cdots\otimes R_0\mathcal{C}(E_r,\delta_r)\ .$$
		We then define 
		$$\Psi_{\mathfrak{d}}:R_0\mathcal{C}(\mathfrak{d})\rightarrow \mathrm{gr}_r^R\mathcal{C}(S,\delta)$$
		by the formula
		$$\Psi_{\mathfrak{d}}(X_0\otimes\cdots\otimes X_r) := \psi_0(X_0)\wedge \omega_{c_1} \wedge \psi_1(X_1)\wedge \omega_{c_2} \wedge \cdots \wedge \omega_{c_r} \wedge \psi_r(X_r)\ .$$

		Let us first prove that~$\Psi_{\mathfrak{d}}$ is well-defined. According to Lemma~\ref{lempullbackformula}, each map~$\psi_i$ sends~$R_0\mathcal{C}(E_i,\delta_i)$ to~$R_0\mathcal{C}(S,\delta)$, hence the term~$\psi_0(X_0)\wedge\cdots\wedge \psi_r(X_r)$ is in~$R_0\mathcal{C}(S,\delta)$. Since the cardinality of~$\mathfrak{d}$ is~$r$, multiplying by~$\omega_{c_1}\wedge\cdots\wedge\omega_{c_r}$ gives an element of~$R_r\mathcal{C}(S,\delta)$.

 	With the same abuse of notation as in Lemma~\ref{lemtechnical}, we claim that we have
		\begin{equation}\label{eqproofPsi}
		\Psi_\di(X_0\otimes\cdots\otimes X_r) = X_0\wedge \omega_{c_1} \wedge X_1\wedge \omega_{c_2} \wedge \cdots \wedge \omega_{c_r} \wedge X_r \;\mod R_{r-1}\mathcal{C}(S,\delta)\ .
		\end{equation}
		
		We do the proof of this equality in the case~$r=1$ and~$\mathfrak{d}=\{c\}$ a chord, the general case being similar and left to the reader. Let us choose monomials~$X_0\in R_0\mathcal{C}(E_0,\delta_0)$,~$X_1\in R_0\mathcal{C}(E_1,\delta_1)$ with zero residual chord. We want to prove the equality
		$$\Psi_{\{c\}}(X_0\otimes X_1)=X_0\wedge\omega_c\wedge X_1 \;\mod R_0\mathcal{C}(S,\delta)\ .$$
		According to Lemma~\ref{lemtechnical}, we may write
		$$\psi_1(X_0)=X_0+\sum_{i_0}X_0^{(i_0)} \textnormal{ and } \psi_1(X_1)=X_1+\sum_{i_1}X_1^{(i_1)}\ ,$$
		where each monomial~$X_0^{(i_0)}$ and~$X_1^{(i_1)}$ has zero residual chord and contains a symbol~$\omega_{c'}$ with~$c'$ crossing~$c$. We can then write the difference~$\Psi_{\{c\}}(X_0\otimes X_1)-X_0\wedge\omega_c\wedge X_1$ as
		$$\sum_{i_0}X_0^{(i_0)}\wedge\omega_c\wedge X_1 + \sum_{i_1}X_0\wedge\omega_c\wedge X_1^{(i_1)} + \sum_{i_0,i_1}X_0^{(i_0)}\wedge\omega_c\wedge X_1^{(i_1)}\ .$$
		All the monomials appearing in the above expression have zero residual chord, hence the result. Equations (\ref{eqproofPhi}) and (\ref{eqproofPsi}) imply that~$\Psi$ is the inverse for~$\Phi$.
		\end{proof}

		\begin{thm}\label{thmfree}
		The dihedral gravity cooperad is cofree. More precisely, it is (non-canonically) isomorphic to the cofree dihedral cooperad on the dihedral module
			$$(S,\delta)\mapsto \det(S)\otimes H^{\bullet+n-3}(\Mdelta_{0,S})(-1)\ .$$
		\end{thm}
		
		\begin{proof}
		It is a consequence of Proposition~\ref{propfreefiltration}, using Lemma~\ref{lemresiduefiltration} and Theorem~\ref{thmmain}, which imply, after operadic suspension, the corresponding statements for the dihedral gravity cooperad. The last statement follows from the last statement of Theorem \ref{propequivalence}.
		\end{proof}
		
		\begin{rem}
		In \cite{dotsenkofreeness}, Dotsenko built a general a criterion to prove the freeness of the nonsymmetric operad underlying an operad in terms of Gr\"{o}bner bases \cite{dotsenkokhoroshkingroebner}. It would be interesting to know whether this criterion can give an alternate proof of Theorem \ref{thmfree}.
		\end{rem}
		
	\subsection{Consequences for Brown's moduli spaces}
	
		We gather here some consequences of Theorem~\ref{thmfree} on the geometry of the moduli spaces~$\Mdelta_{0,S}$.
	
		\begin{coro}\label{coropurity}
		For every integer~$k$ and every dihedral set~$(S,\delta)$, the cohomology group~$H^k(\Mdelta_{0,S})$ is pure Tate of weight~$2k$.		
		\end{coro}
		
		\begin{proof}
		This follows from Theorem~\ref{thmfree} and Theorem~\ref{propequivalence}.
		\end{proof}
		
		\begin{coro}\label{coroinj}
		For every integer~$k$ and every dihedral set~$(S,\delta)$, the natural map~$H^k(\Mdelta_{0,S})\rightarrow H^k(\M_{0,S})$ is injective and fits into a long exact sequence
		\begin{equation}\label{eqlongexseccohomology}
		0\rightarrow H^k(\Mdelta_{0,S}) \rightarrow H^k(\M_{0,S}) \rightarrow \bigoplus_{\di\in\mathsf{Diss}_1(S,\delta)}H^{k-1}(\M(\di))(-1)\rightarrow \bigoplus_{\di\in\mathsf{Diss}_2(S,\delta)} H^{k-2}(\M(\di))(-1) \rightarrow \cdots \ .
		\end{equation}
		\end{coro}
		
		\begin{proof}
		By Theorem~\ref{thmfree} and the proof of Theorem~\ref{propequivalence}, we get an injective map~$H^k(\Mdelta_{0,S})\rightarrow E_1^{0,k}=H^k(\M_{0,S})$. By the construction of the residue spectral sequence, this map is indeed the one induced in cohomology by the inclusion~$\M_{0,S}\hookrightarrow\Mdelta_{0,S}$.
		\end{proof}
		
		We note that the image of the natural map~$H^{\bullet-1}(\Mdelta_{0,S})(-1)\hookrightarrow \mathcal{C}(S,\delta)$ is exactly the subspace~$R_0\mathcal{C}(S,\delta)$.\\
		
		Let us recall that the Betti numbers of the spaces~$\M_{0,n}$ are given by the Poincar\'{e} polynomials
		$$\sum_{k=0}^{n-3}b_k(\M_{0,n})\, x^k=\prod_{j=2}^{n-2}(x-j)\ .$$	
		By taking the Euler characteristic of the exact sequence (\ref{eqlongexseccohomology}), one may thus derive a formula for the Betti numbers of the spaces~$\Mdelta_{0,n}$, as follows.
		
		\begin{coro}\cite{bergstrombrown}\label{corobetti}
		The generating series
		$$f(x,t)=x-\sum_{n\geq 3}\left(\sum_{k=0}^{n-3}(-1)^k\,b_k(\M_{0,n})\,t^{n-3-k}\right)x^{n-1}$$
		$$f^{\delta}(x,t)=x+\sum_{n\geq 3} \left(\sum_{k=0}^{n-3}(-1)^k\,b_k(\Mdelta_{0,n})\,t^{n-3-k}\right)x^{n-1}$$
		are inverse one to another :~$f(f^\delta(x,t),t)=f^\delta(f(x,t),t)=x$.
		\end{coro}		 
		 
		 We note that in~\cite[Section 3]{bergstrombrown}, the injectivity statement of Corollary~\ref{coroinj} is used but not proved. 		
		 
		\begin{coro}\label{coroformality}
		For every dihedral set~$(S,\delta)$, Brown's moduli space~$\Mdelta_{0,S}$ is a formal topological space.
		\end{coro}
		
		\begin{proof}
		It is a consequence of Corollary~\ref{coropurity} and~\cite[Theorem 2.5]{dupontformality}. A more direct proof goes as follows. Recall that~$\Mdelta_{0,S}$ is a smooth affine complex variety. We denote by~$\Omega^\bullet(\Mdelta_{0,S})$ the complex of global holomorphic differential forms on~$\Mdelta_{0,S}$, and by~$\Omega^\bullet(\Mdelta_{0,S},\log\partial\Mdelta_{0,S})$ the complex of global holomorphic logarithmic differential forms on~$\Mdelta_{0,S}$ along~$\partial\Mdelta_{0,S}$. Let us recall that the morphism~$H^\bullet(\M_{0,S}) \rightarrow \Omega^\bullet(\Mdelta_{0,S},\log\partial\Mdelta_{0,S})$ which maps the class of~$\omega_c$ to~$\omega_c$ is well-defined and is a quasi-isomorphism.
		We consider the commutative diagram
		$$
		\xymatrix{
		0 \ar[r] & H^\bullet(\Mdelta_{0,S}) \ar[r] & H^\bullet(\M_{0,S}) \ar[r]\ar[d] & \bigoplus_{\mathfrak{d}\in \mathrm{Diss}(S,\delta)} H^{\bullet-1}(\M(\mathfrak{d}))(-1) \ar[d] \\
		0 \ar[r] & \Omega^\bullet(\Mdelta_{0,S}) \ar[r] & \Omega^\bullet(\Mdelta_{0,S},\log\partial\Mdelta_{0,S}) \ar[r] & \bigoplus_{\mathfrak{d}\in \mathrm{Diss}(S,\delta)} \Omega^{\bullet-1}(\Mdelta(\mathfrak{d}),\log\partial\Mdelta(\mathfrak{d}))
		}$$
		where all arrows are morphisms of cochain complexes and where the vertical arrows are quasi-isomorphisms. The first row is exact by Corollary~\ref{coroinj}; the exactness of the second row follows from the fact that a logarithmic differential form on~$\Mdelta_{0,S}$ along~$\partial\Mdelta_{0,S}$ is regular on~$\Mdelta_{0,S}$ if and only if its residue along each~$\Mdelta(\mathfrak{d})$ is zero. Completing the diagram gives the following quasi-isomorphism, hence the result:
		$$H^\bullet(\Mdelta_{0,S}) \rightarrow \Omega^\bullet(\Mdelta_{0,S})\, .$$
		\end{proof}
		
	\subsection{The dihedral Lie operad is free}		
		
		As a corollary of Theorem~\ref{thmfree} and in view of Theorem~\ref{thmlie}, we get a geometric proof of a dihedral enhancement of the theorem of Salvatore and Tauraso about the nonsymmetric Lie operad~\cite{salvatoretauraso}.
		
		\begin{coro}\label{coroliefree}
		The dihedral Lie operad is free. More precisely, it is (non-canonically) isomorphic to the free dihedral operad on the dihedral module
		$$(S,\delta)\mapsto \det(S)\otimes H_{n-3}(\Mdelta_{0,S})(1).$$
		\end{coro}
		
\begin{rem}		The equality between the top Betti number of~$\Mdelta_{0,n}$ and the number of generators of the non-symmetric Lie operad in arity~$n$ in~\cite{salvatoretauraso} was already noticed in~\cite{bergstrombrown}.
\end{rem}

\bibliographystyle{alpha}
\bibliography{biblio}

\end{document}